\documentclass[twocolumn]{autart}    
\usepackage{amsmath}
\usepackage{amssymb}
\usepackage{amsfonts}
\usepackage{mathtools}
\usepackage{bm}
\usepackage{needspace}

\usepackage{booktabs}  
\usepackage{tabularx}  
\usepackage{graphicx}
\usepackage{cite}

\usepackage{algorithm}
\usepackage{algorithmic}
\makeatletter
\@ifundefined{AND}{}{}
\makeatother

\usepackage{wrapfig}

\newenvironment{biography}[2]{
	\par\addvspace{10pt}\noindent
	\begin{wrapfigure}{l}{15mm} 
		\vspace{-10pt} 
		\includegraphics[width=21mm,height=27mm,clip,keepaspectratio]{#2}
	\end{wrapfigure}
	\textbf{#1}
}{%
	\par\addvspace{15pt} 
}

\newtheorem{theorem}{Theorem}[section]
\newtheorem{lemma}{Lemma}[section]

\newtheorem{definition}{Definition}[section]
\newtheorem{assumption}{Assumption}[section]

\newtheorem{remark}{Remark}[section]

\makeatletter
\@ifundefined{proof}{%
	\newenvironment{proof}{\par\noindent\textit{Proof.}\ }{\hfill$\blacksquare$\par}%
}{}%
\makeatother

\begin{document}
	
	\begin{frontmatter}

		\title{Loopless Proximal Riemannian Gradient EXTRA for Distributed Optimization on Compact Manifolds}
		\thanks[footnoteinfo]{This work was supported in part by the National Natural Science Foundation
			of China (62203254). \\	
		\thanks[CorrespondingAuthor] \textsuperscript{*} Corresponding author. }
		\author[Paestum]{Yongyang Xiong}\textsuperscript{,*}\ead{xiongyy25@mail.sysu.edu.cn},  
			\author[Paestum]{Chen Ouyang}\ead{ouych26@mail2.sysu.edu.cn},              %
			\author[tsinghua]{Keyou You}\ead{youky@tsinghua.edu.cn},
				\author[vu]{Yang Shi}\ead{yshi@uvic.ca},    
			\author[hit]{Ligang Wu}\ead{ligangwu@
				hit.edu.cn} 
				
		\address[Paestum]{School of Intelligent Engineering, Sun Yat-Sen University, Shenzhen 518107, China}  
				\address[tsinghua]{Department of Automation, Tsinghua University, Beijing 100084, China} 
					\address[vu]{Department of Mechanical Engineering, University of
						Victoria, Victoria, BC V8W 2Y2, Canada}  
				\address[hit]{The Key
					Laboratory of Autonomous Intelligent Unmanned Systems, Harbin Institute
					of Technology, Harbin 150001, China}

		\begin{abstract}
			
	Distributed optimization has gained substantial interest in recent years due to its wide applications in machine learning. However, most of existing algorithms are designed for Euclidean spaces, leaving composite optimization on Riemannian manifolds largely unexplored. To bridge this gap, we propose the proximal Riemannian gradient EXTRA algorithm (PR-EXTRA) to solve distributed composite optimization problem with nonsmooth regularizer over compact manifolds. In each iteration, PR-EXTRA requires only a single round communication, coupled with local gradient evaluations and proximal mappings. Furthermore, a manifold projection operator is integrated to ensure the feasibility of all iterates throughout the optimization process. Theoretical analysis shows that with a constant stepsize, PR-EXTRA achieves a sublinear convergence rate of \( \mathcal{O}\bigl(1/K\bigr) \) to a stationary point, matching the proximal gradient EXTRA algorithm in Euclidean spaces. Numerical experiments show the effectiveness of the proposed algorithm.	
		\end{abstract}

	\end{frontmatter}

	\section{Introduction}
	
	Distributed optimization plays a pivotal role in large-scale systems, providing the essential framework for decentralized decision-making in complex environments. This paradigm is particularly critical in domains such as distributed computing \cite{Yu2014,LHQ2024,You2025}, federated learning \cite{McMahan2017,Li2024}, and sensor networks \cite{Niyato2021,Bertrand2024}, where data are naturally generated and stored at disparate locations. By eliminating the necessity for a central coordinator, distributed approaches effectively mitigate communication bottlenecks and address the escalating privacy concerns associated with centralized data aggregation. The primary objective in these settings is for networked nodes to cooperatively minimize the sum of their local objective functions via local communications.
	%
	
	Distributed optimization in Euclidean spaces is underpinned by an extensive body of research \cite{Li2022,Li2021,Pu2024,Wu2026}. At the heart of this discipline lies the distributed (sub)gradient descent (DGD) framework, which synergizes local updates with inter-node consensus \cite{Nedic2009}. Despite its prevalence, DGD with constant stepsizes is inherently restricted to converging within a steady-state neighborhood, failing to achieve exact stationarity. To address this limitation, gradient tracking stands as a versatile paradigm that explicitly tracks global gradient directions to eliminate such errors. This strategy demonstrates high adaptability across various complex scenarios, including stochastic, quantized, and time-varying environments \cite{Shi2025,Pu2022,Vikalo2025,xiong2023}. Exact convergence can also be achieved via bias-correction strategies such as EXTRA \cite{Shi2015,Touri2017,Moulines2017,Ying2018}, which utilize historical information to eliminate the steady-state error. These algorithms are particularly noteworthy for their communication efficiency, as they often incur a lower per-iteration overhead compared to gradient tracking while maintaining linear convergence under strong convexity.
	
	In contrast, distributed optimization over Riemannian manifolds is an emerging yet rapidly advancing frontier \cite{Boumal2023}. This shift is driven by applications where data naturally reside on structured geometries, such as orthogonality constraints in PCA or rank constraints in low-rank matrix completion \cite{Wang2023}. Extending algorithms from Euclidean to manifold settings introduces non-trivial challenges due to the inherent nonconvexity and nonlinearity of manifold constraints. A primary obstacle is the absence of a global vector space structure, which renders standard linear consensus protocols inapplicable. Since local iterates reside on a curved space, simple weighted averaging often leads to points that deviate from the manifold. Furthermore, the alignment of gradient information across nodes is significantly complicated by the fact that each node operates within its own unique tangent space. This discrepancy requires the use of computationally intensive operators, such as parallel transports or retractions, to facilitate meaningful communication and gradient aggregation. Early Riemannian consensus algorithms, primarily relying on geodesic distances as explored in the work \cite{Vidal2013}, often incur heavy computational demands from exponential mapping. To address this issue, the work \cite{Chen2021} introduces the distributed Riemannian gradient descent algorithm that utilizes retraction operators. To further alleviate the computational burden of manifold constraints, the work \cite{Chen2024} considers consensus over the Stiefel manifold using Euclidean distances. Notably, the work \cite{Deng2025} proposed a projection-based distributed Riemannian gradient descent algorithm. This algorithm simplifies the handling of manifold constraints while preserving fast convergence, effectively bridging the gap between theoretical geometric rigor and practical algorithmic efficiency. Along this line of research, Riemannian EXTRA \cite{REXTRA} introduces a communication-efficient method that achieves network consensus with only a single round of communication per iteration. Beyond these fundamental contributions, the scope of Riemannian distributed optimization has recently expanded to encompass more complex problem settings. For instance, recent innovations introduce variance reduction algorithms \cite{Lei2026} and personalized algorithms \cite{Xu2025ACC} to enhance efficiency and robustness in stochastic or heterogeneous scenarios.
	
	Despite these advances, existing Riemannian distributed algorithms predominantly focus on smooth objectives. The broader class of composite optimization on manifolds, which involves objectives with both smooth and nonsmooth components, remains significantly under-explored in distributed settings. These problems are inherently nonsmooth and nonconvex. In such contexts, nonsmooth regularizers promote desirable structural properties while manifold constraints encapsulate the essential underlying geometry. The work in \cite{Deng2024} introduces a distributed Riemannian subgradient algorithm for nonsmooth problems, while the work in \cite{Wang2025} proposes a distributed proximal gradient tracking algorithm that achieves a convergence rate of $\mathcal{O}(1/{\sqrt{K}})$. As a compelling alternative, EXTRA-type frameworks offer communication efficiency by exchange only local iterates. Note that the work \cite{Shi-2015} extends EXTRA to composite optimization in Euclidean spaces. However, the corresponding work on Riemannian manifolds with regularization terms remains largely unexplored. This motivates us to investigate whether such an efficient framework with low communication complexity can be generalized to composite optimization over Riemannian manifolds, with the goal of developing effective distributed algorithms for geometrically constrained settings?
	
	To bridge this gap, this paper proposes PR-EXTRA for distributed composite optimization over compact Riemannian manifolds. To handle nonsmooth regularizers, the algorithm applies a Riemannian proximal operator while performing all iterative updates via computationally efficient projection operators. Particularly, PR-EXTRA achieves exact convergence with only a single round of communication per iteration. The primary contributions of this work are as follows:
	
	(1) Algorithmically, we propose PR-EXTRA, a distributed loopless proximal gradient EXTRA algorithm for composite optimization over the manifold. In comparison with the distributed Riemannian optimization algorithms \cite{Chen2024,Wang2025}, the proposed algorithm requires only single-round consensus and proximal operations for nonsmooth terms, reducing the computational and communication overheads at each node.
	
	(2) Theoretically, we prove that PR-EXTRA achieves a sublinear convergence rate of $ \mathcal{O}\left(1/{K}\right) $. This matches the convergence rate of proximal gradient EXTRA \cite{Shi-2015} in Euclidean spaces. The effectiveness of the proposed algorithm is verified through numerical experiments.
	
	The rest of this paper is organized as follows. Section \ref{sec2} states the problem. The proposed PR-EXTRA
	is presented in Section \ref{sec3}. Section \ref{sec4} provides convergence analysis. Section \ref{sec5} shows numerical results. We conclude this
	paper in Section \ref{sec6}.

	\textit{Notations.} Let $\mathcal{M} \subset \mathbb{R}^{d \times r}$ denote a compact submanifold endowed with the Euclidean inner product $\langle \cdot, \cdot \rangle$ and its induced norm $\|\cdot\|$. We use $\mathcal{M}^n$ to denote its $n$-fold Cartesian product. For $x \in \mathcal{M}$, $T_x \mathcal{M}$ and $N_x \mathcal{M}$ stand for the tangent and normal spaces, respectively. We define $[n] := \{1, \dots, n\}$ and $J := \frac{1}{n} \mathbf{1}_n \mathbf{1}_n^\top$, where $\mathbf{1}_n$ denotes the all-ones vector. The notion $\otimes$ denotes the Kronecker product, and $\mathbf{W} := W \otimes I$. We stack local variables $x_i$ into a matrix $\bm{x} := [x_1^\top, \dots, x_n^\top]^\top \in \mathbb{R}^{nd \times r}$ and denote their average as $\hat{x} := \frac{1}{n} \sum_{i=1}^n x_i$. The proximal mapping of $h$ is defined by $\text{prox}_{\tau h}(x) := \arg\min_{y} \{ h(y) + \frac{1}{2\tau} \|y - x\|^2 \}$. The notation $y = \mathcal{O}(x)$ means that there exists a positive constant $M$ such that $y \leq Mx$. We define the distance between $x$ and $\mathcal{M}$ as $\mathrm{dist}(x,\mathcal{M}) := \text{inf}\;\{\|y-x\|:y\in\mathcal{M}\}$, and the projection of $x$ onto $\mathcal{M}$ as $\mathcal{P}_{\mathcal{M}}(x) := \arg\min_{y \in \mathcal{M}} \|y - x\|$. For a differentiable function $h$, let $\nabla h(x)$ denote the Euclidean gradient. We define the Riemannian gradient of $h$ as $\operatorname{grad} h(x) := \mathcal{P}_{T_x \mathcal{M}}(\nabla h(x))$.
	
		\section{Preliminaries and Problem Formulation}\label{sec2}
	In this section, we first introduce some basic concepts of graph theory. Then, we formalize the distributed Riemannian manifolds optimization problem of interest.
	\subsection{Graph Theory}
	Consider an undirected graph $\mathcal{G}(\mathcal{V}, \mathcal{E})$ where $\mathcal{V}=\{1, \dots, n\}$ represents the set of nodes, and $\mathcal{E} \subseteq \mathcal{V} \times \mathcal{V}$ denotes the set of undirected links. 
	An edge $(j, i) \in \mathcal{E}$ implies that node $i$ can receive information from node $j$. 
	The neighbor set of node $i$ is denoted by $\mathcal{N}_i = \{j : (j, i) \in \mathcal{E} \}$. 
	Note that $(i, j) \in \mathcal{E}$ implies $(j, i) \in \mathcal{E}$. 
	The graph $\mathcal{G}(\mathcal{V}, \mathcal{E})$ is connected if there exists a path between any pair of distinct nodes. 
	The interaction strengths among nodes are encoded in a weight matrix $W = [w_{ij}] \in \mathbb{R}^{n \times n}$ with $w_{ij} > 0$ if $(j, i) \in \mathcal{E}$, and $w_{ij} = 0$ otherwise. 
	We say $W$ is doubly stochastic if $W \mathbf{1}_n = \mathbf{1}_n$ and $W^\top \mathbf{1}_n = \mathbf{1}_n$. We make the following assumption throughout this paper, which is commonly adopted in the literature of distributed Riemannian optimization~\cite{REXTRA,Deng2025}.
	\begin{assumption}\label{ass1}
		The undirected graph \(G\) is is fixed and connected.
	\end{assumption}
	
	\subsection{Problem Formulation}
	Consider a distributed optimization problem over an undirected network of $n$ nodes. Each node $i \in \mathcal{V} = \{1, \dots, n\}$ is associated with a private smooth local cost function $f_i: \mathbb{R}^{d \times r} \to \mathbb{R}$. In addition to the local objectives, all nodes share a common convex regularizer $r: \mathbb{R}^{d \times r} \to \mathbb{R}$. We focus on the following composite optimization problem over the manifold:
	\begin{align}\label{composite}
		&\underset{x \in \mathcal{M}}{\text{min}}  \; h(x) = \frac{1}{n} \underbrace{\sum_{i=1}^n f_i(x)}_{f(\bm{x})} + r(x),
	\end{align}
	where $\mathcal{M} \subset \mathbb{R}^{d \times r}$ \cite{manifolds} is a compact smooth Riemannian manifold. An example of a manifold $\mathcal{M}$ is the Stiefel manifold, which is defined as $\mathrm{St}(d,r) := \{ x \in \mathbb{R}^{d \times r} : x^\text{T} x = I_r \}$. This structure is used in many areas, such as principal component analysis \cite{pca}, low-rank matrix completion \cite{lr}, and deep neural networks with orthogonality constraints \cite{network_constraint}.
	Throughout this paper, we make the
	following assumptions on the objective functions:
	\begin{assumption}\label{ass2}
		Each $f_i$ is $L_f$-smooth over the convex hull of $\mathcal{M}$ and its Euclidean gradient is bounded by $L_g$ on $\mathcal{M}$, i.e., 
		\[
		\|\nabla f_i(x) - \nabla f_i(y)\| \le L_f \|x - y\|, \quad \|\nabla f_i(z)\| \le L_g.
		\]
	\end{assumption}
	Under Assumption \ref{ass2}, it follows from \cite{Deng2025} that there exists a constant $L = \max\{ L_f + L_g/R, L_g, L_f + L_g L_{p} \}$, where $L_{p}$ is the Lipschitz constant of $\mathcal{P}_{T_x\mathcal{M}}$ over the convex hull of $\mathcal{M}$, such that
	\begin{align*}
		f_i(y) -f_i(x) -\langle \mathrm{grad} f_i(x), y - x \rangle &\le  \frac{L}{2}\|y - x\|^2,  \\ 
		\|\mathrm{grad} f_i(x)  - \mathrm{grad} f_i(y)\| &\le L \|x - y\|. 
	\end{align*} 
	
	\begin{assumption}\label{ass3}
		The regularizer $r$ is convex and  $L_r$-continuous. 
	\end{assumption}
	
	Assumption \ref{ass2} and \ref{ass3} are standard in the Riemannian optimization literature. Specifically, the Lipschitz-smooth condition in Assumption \ref{ass2} are widely adopted in smooth manifold optimization \cite{REXTRA, Deng2025,Chen2024}, while Assumption \ref{ass3} on the convex regularizer is typical for composite optimization settings, even in Euclidean spaces \cite{Shi-2015}.
	
	With the problem \eqref{composite} and associated assumptions available, we present its stationarity condition of the problem \eqref{composite}. 
	\begin{definition}\label{opt}
		A point \( x \in \mathcal{M} \) is called a stationary point of the problem \eqref{composite} if it satisfies the following stationarity condition,
		\[
		0 \in \mathcal{P}_{T_{x}\mathcal{M}}\bigl(\nabla f(x) + \partial r(x)\bigr).
		\]
	\end{definition}
	\begin{definition}
		For a given \( \epsilon > 0 \), the set of points \{$x_i$\} is an \(\epsilon\)-stationary solution to problem \eqref{composite} if the \eqref{d2eq} is satisfied for for all $i\in\mathcal{V}$,
		\begin{equation}\label{d2eq}
			\begin{cases}
				\| x_i - \bar{x} \| \leq \epsilon, \\[6pt]
				\text{dist}(0,\mathcal{P}_{T_{x_i}\mathcal{M}}(\nabla f(x_i) + \partial r(x_i))  \leq \epsilon,
			\end{cases}
		\end{equation}
		where $\bar{x} \in \mathcal{P}_{\mathcal{M}}(\hat x)$.
			\end{definition} 
			As $\epsilon$ tends to  0, this definition recovers the exact stationarity in Definition \ref{opt}. Notably, the above deﬁnition are also adopted in \cite{Chen2022, Wang2025}.

	\section{Algorithm Development}\label{sec3}
In this section, we begin by introducing the proximal gradient EXTRA algorithm in Euclidean spaces \cite{Shi-2015} and discuss the challenges in extending it to Riemannian manifolds. Then, we propose the PR-EXTRA.
\subsection{Proximal gradient EXTRA Algorithm}
PG-EXTRA extends the EXTRA \cite{Shi2015} to handle nonsmooth regularization terms with a fixed stepsize. Each node $i$ maintains two variables $x_{i,k}$ and $y_{i, k}$, the update rules of PG-EXTRA are given by
\begin{subequations}\label{pg-extra}
	\begin{align}
		y_{i, k+1} &= y_{i, k} - \alpha \big[ \nabla f_i(x_{i, k+1}) - \nabla f_i(x_{i, k}) \big] \nonumber \\
		&\quad + \sum_{j =1}^n w_{ij} x_{j, k+1} - \sum_{j =1}^n \tilde{w}_{ij} x_{j, k}, \label{pg-extra-a} \\
		x_{i, k+2} &= \underset{x}{\text{argmin}}\; r(x)+\frac{1}{2\alpha}\|x-	y_{i, k+1}\|^2, \label{pg-extra-b}
	\end{align}
\end{subequations}
where $\tilde w_{ij}$ is the $(i,j)$-th entry of the mixing matrix $\tilde{W} := \frac{1}{2}(I + W)$ and $\alpha$ is the stepsize. Equation \eqref{pg-extra} integrates the consensus, gradient correction, and the proximal mapping for the nonsmooth term $r$. To provide deeper insights into the convergence mechanism of \eqref{pg-extra} and facilitate its extension to manifolds, we present a decoupled structure in \eqref{pg-extra-decoupled-discrete}. By introducing an auxiliary variable $s_{i,k}$ to capture the cumulative residuals of consensus and gradients, we decouple the proximal operation from the \eqref{pg-extra}. Consequently, the \eqref{pg-extra} can be expressed as
\begin{subequations} \label{pg-extra-decoupled-discrete}
	\begin{align}
		x_{i, k+1} &= \underset{x}{\text{argmin}} \; r(x)+\frac{1}{2\alpha}\|x-	y_{i, k}\|^2 \label{pg-extra-A},\\
		s_{i, k+1} &= s_{i, k} +  \sum_{j =1}^n (w_{ij} - \tilde{w}_{ij}) x_{j, k}\nonumber\\
		&\quad - \alpha \big[ \nabla f_i(x_{i, k+1}) - \nabla f_i(x_{i, k}) \big] \label{pg-extra-B}, \\
		y_{i, k+1} &=  \sum_{j =1}^n w_{ij} x_{j, k+1} + s_{i, k+1}, \label{pg-extra-c}
	\end{align}
\end{subequations}
where the variable $s_{i,k+1}$ in \eqref{pg-extra-B} serves as a correction term that accumulates historical information to correct the local gradient direction, ensuring that the nodes reach exact convergence. The update $x_{i,k+1}$ in \eqref{pg-extra-A} is obtained by applying the proximal operator of the regularizer $r$ to $y_{i,k}$. Finally, \eqref{pg-extra-c} updates $y_{i,k+1}$ by aggregating neighborhood information $\sum_{j =1}^n w_{ij} x_{j, k+1}$ and incorporating the correction term $s_{i,k+1}$. PG-EXTRA effectively neutralizes the steady-state bias prevalent in  distributed algorithms, achieving an convergence rate of $\mathcal{O}(1/K)$ in terms of the first-order optimality residual \cite{Shi-2015}. However, extending PG-EXTRA to Riemannian manifolds encounters two fundamental theoretical and computational obstacles. Firstly, the PG-EXTRA relies on global vector-space operations, such as linear combinations, gradient differences. These operations are ill-defined on curved manifolds because tangent spaces at distinct points are disjoint and cannot be directly compared. Secondly, the proximal operator defined in Euclidean spaces becomes computationally intractable when adapted to the manifold setting using squared geodesic distance. This complexity arises primarily from the difficult combination of the nonsmooth nature of the objective function \eqref{composite} and the nonconvexity of manifold constraints. To address these issues, we propose the PR-EXTRA in the next subsection.

\subsection{Proximal Riemannian gradient EXTRA Algorithm}
Consider the nonlinear geometry of the Riemannian manifold $\mathcal{M}$, where Euclidean gradients typically do not lie in the tangent space, rendering them unsuitable for optimization. To address the inherent steady-state bias arising from distributed settings over the manifolds, we need to substitute Euclidean gradients $\nabla f_i$  with Riemannian gradient $\text{grad} f_i$ in \eqref{pg-extra-B},
\begin{equation}\label{sk}
	\begin{aligned}[b]
		s_{i,k} = s_{i,k-1} &+ \sum_{j=1}^n (w_{ij} - \tilde{w}_{ij}) x_{j,k-1}  \\
		&- \alpha[\operatorname{grad} f_i(x_{i,k}) - \operatorname{grad} f_i(x_{i,k-1})].
	\end{aligned}
\end{equation}
The variable $s_{i,k}$ in \eqref{sk} accumulates historical Riemannian gradient to correct the local descent direction, thereby enabling each node to converge to a stationary point. Subsequently, in \eqref{yk},  we compute $y_{i,k}$ by aggregating neighboring information and applying the correction term $s_{i,k}$, using a projection operator to ensure that the variable $y_{i,k}$ remains on the manifold.
\begin{align}
	y_{i,k} &= \mathcal{P}_{\mathcal{M}} \left( \sum_{j=1}^n w_{ij} x_{j,k} + s_{i,k} \right)\label{yk}.
\end{align}

Inspired by the work \cite{Shi-2015} and \cite{Chen2022}, we need to adapt the proximal operator to the Riemannian manifold to handle the regularizer $r$ in \eqref{composite}. We adopt the Riemannian proximal gradient algorithm for centralized manifold optimization \cite{Chen2022} and determine the descent direction $\eta_i$ bysolving a minimization subproblem on the tangent space $T_{x_{i,k}} \mathcal{M}$,
\begin{align}\label{subproblem}
	\underset{{\eta_i} \in T_{x_{i,k}} \mathcal{M}}{\text{argmin}} \; \left\langle \nabla f_i(x_{i,k}), \eta_i \right\rangle + \frac{1}{2\tau} \|\eta_i\|^2 + r(x_{i,k} + \eta_i),
\end{align}
where $\tau > 0$ is the stepsize. This subproblem \eqref{subproblem} effectively seeks an update in the tangent space that balances the first-order descent of $f$  with the structural constraints imposed by $r$. To improve computational efficiency, we consider a setting where the Riemannian proximal operator is applied solely to the regularizer $r$ associated with \eqref{pg-extra-A} over the manifold,
\begin{equation}\label{nk}
	\eta_{i,k} = \underset{\eta_i \in T_{y_{i,k}} \mathcal{M}}{\text{argmin}} \;  \frac{1}{2\tau} \|\eta_i\|^2 + r(y_{i,k} + \eta_i).
\end{equation}

Finally, we update $x_{i,k+1}$ by leveraging the variable $y_{i,k}$ and the descent direction $\eta_{i,k}$, 
\begin{align}
	x_{i,k+1} &= \mathcal{P}_{\mathcal{M}}( y_{i,k} +\eta_{i,k}).\label{xk}
\end{align}

The proposed algorithm proceeds by iteratively aggregating neighbor information via a projection-based consensus on the manifold, solving a Riemannian proximal subproblem to compute a descent direction in the tangent space, and applying a correction step to compensate for historical Riemannian gradient errors. We present the detailed description of PR-EXTRA in Algorithm \ref{algorithm_extra}. To facilitate the convergence analysis, we reformulate PR-EXTRA into the following compact form,
\begin{subequations}
	\begin{align}
		\bm{y}_{k} &= \mathcal{P}_{\mathcal{M}^n}( \bm{W}\bm{x}_k +\bm{s}_k)\\
		\bm{\eta}_k &=\underset{\bm{\eta} \in T_{\bm{y}_{k}} \mathcal{M}^n}{\text{argmin}} \;  \frac{1}{2\tau} \|\bm{\eta}\|^2 + R(\bm{y}_{k} + \bm{\eta}),\\
		\bm{x}_{k+1} &= \mathcal{P}_{\mathcal{M}^n}( \bm{y}_{k} +\bm{\eta}_k),\\
		\bm{s}_{k+1}
		&= \bm{s}_{k}- \alpha \big[\operatorname{grad} f(\bm{x}_{k+1}) - \operatorname{grad} f(\bm{x}_{k})\big] \notag
		\\
		& \quad + (\mathbf{W} - \tilde{\mathbf{W}}) \bm{x}_{k} .
	\end{align}
\end{subequations}
where $R(\bm{y}_{k} +\bm{\eta}_k)=\sum_{i=1}^n r(y_{i,k}+\eta_{i,k})$.

\begin{algorithm}[t]
	\caption{Proximal Riemannian gradient EXTRA (PR-EXTRA)}
	\label{algorithm_extra}
	\begin{algorithmic}[1]
		\REQUIRE Initial point $x_0 \in \mathcal{N}$, stepsize $\alpha, \tau > 0$, $s_0 = -\alpha \operatorname{grad} f(x_0)$, for each $i\in \mathcal{V}$.
		\FOR{$k = 0, 1, 2, 3, \dots$}
		
		\STATE Riemannian EXTRA step update $s_{i,k}$ according to \eqref{sk}.\label{Line1}
		
		\STATE Consensus update $y_{i,k}$ according to \eqref{yk}.\label{Line2}
		
		\STATE Proximal step update $\eta_{i,k}$ according to \eqref{nk}.\label{Line3}
		
		\STATE Update $x_{i,k+1}$ according to \eqref{xk}.		\label{Line4}
		
		\STATE $k = k+1$.
		\ENDFOR	
	\end{algorithmic}
\end{algorithm}

\begin{remark}
	The proposed PR-EXTRA ensures high communication efficiency through a single round of neighbor communication per iteration. To achieve global consensus, Eq. \eqref{yk} aggregates local information across the network, while Eq. \eqref{nk} explicitly handles the nonsmooth regularization via a Riemannian proximal mapping. Convergence to a stationary point is rigorously guaranteed by the historical correction term introduced in \eqref{sk}. These components are synthesized in the final update step \eqref{xk}. Furthermore, the theoretical analysis in the subsequent section shows that the convergence rate of PR-EXTRA over the manifold matches that of PG-EXTRA \cite{Shi-2015} in the Euclidean space.
\end{remark}

	\section{Convergence Analysis}\label{sec4}
In this section, we provide a theoretical guarantee for PR-EXTRA. Firstly, to facilitate the consensus analysis on the Riemannian manifold, we introduce a quadratic consensus problem. Subsequently, we prove the boundedness of the sequence generated by PR-EXTRA and establish a sufficient descent property for problem \eqref{composite}. Finally, we prove the sublinear convergence rate of the proposed algorithm.

For ease of notation, we define
\begin{align*}
	&\operatorname{grad} f(\bm{x}_{k}) = \left[ \operatorname{grad} f_{1}(x_{1,k})^{\text{T}}, \cdots, \operatorname{grad} f_{n}(x_{n,k})^{\text{T}} \right]^{\text{T}}, \notag \\ 
	&\hat{g}_{k}:= \frac{1}{n} \sum_{i=1}^{n} \operatorname{grad} f_{i}(x_{i,k}), \quad 
	\hat{\bm{g}}_{k}:= (\mathbf{1}_{n} \otimes I_{d}) \hat{g}_{k}.
\end{align*}
The convergence analysis relies on certain smoothness properties of the projection operator,  which are presented below.

For any constant $\tau > 0$, the $\tau$-tube around $\mathcal{M}$ as the set
\[
\bar{U}_{\mathcal{M}}(\tau) := \{ x : \mathrm{dist}(x,\mathcal{M}) \le \tau \}.
\]
A closed set $\mathcal{M}$ is said to be $R$-proximally smooth if the projection 
$\mathcal{P}_{\mathcal{M}}(x)$ is a singleton whenever $\mathrm{dist}(x,\mathcal{M}) < R$. 
For a constant $\gamma \in (0,R)$, a $R$-proximally smooth set $\mathcal{M}$ satisfies that
\begin{equation}\label{R-proximal}
	\begin{aligned}
		\|\mathcal{P}_{\mathcal{M}}(x) - \mathcal{P}_{\mathcal{M}}(y)\| 
		\le \frac{R}{R - \gamma} \|x - y\|,
		\forall x, y \in \bar{U}_{\mathcal{M}}(\tau).
	\end{aligned}
\end{equation}
In particular, $\mathcal{P}_{\mathcal{M}}$ is asymptotic $1$-Lipschitz as $\gamma$ tends to $0$. It is known from the \cite{clarke} that any compact \(C^2\) submanifold of Euclidean space is proximally smooth. Throughout this paper, we assume that the manifold \(\mathcal{M}\) in problem \eqref{composite} is \(R\)-proximally smooth for some \(R > 0\), the projection operator satisfies the following result.
\begin{lemma}[\cite{Deng2025}] \label{L1}
	Given an $R$-proximally smooth compact submanifold $\mathcal{M}$, for any $x \in \mathcal{M}$, 
	$u \in \{ u \in \mathbb{R}^{d\times r} : \|u\| \le \frac{R}{2} \}$, 
	there exists a constant $Q > 0$ such that
	\begin{equation}\label{L1-eq}
		\begin{aligned}
			\|\mathcal{P}_{\mathcal{M}}(x + u) - x - \mathcal{P}_{T_x \mathcal{M}}(u)\| \le Q \|u\|^2.
		\end{aligned}
	\end{equation}
\end{lemma}

Another useful technical lemma is also provided, which bounds the distance between the Euclidean mean and the manifold mean.
\begin{lemma}[\cite{Chen2024}] \label{l3} For any $\bm{x} \in \mathcal{M}^n$ satisfying $\|x_i - \bar{x}\| \leq \delta$, there exists a constant $M > 0$ such that the following inequality holds
	\begin{align*}
		\|\bar{x} - \hat{x}\| \leq M \frac{\| \bm{x} - \bar{\bm{x}} \|^2}{n}. 
	\end{align*}
\end{lemma}
To analyze the consensus of the PR-EXTRA, we introduce the consensus problem \eqref{consensusproblem}. Minimizing the problem \eqref{consensusproblem} implies that all nodes achieve optimal consensus over the Riemannian manifold.
\begin{align}\label{consensusproblem}
	\underset{\bm{x} }{\text{min}}  \; \phi(\bm{x})= \frac{1}{4}\sum_{i=1}^{n} \sum_{j=1}^{n} w_{ij} \|x_i-x_j\|^2,
\end{align}
where the Euclidean gradient of $\phi(\bm{x})$ is given by
$\nabla \phi(\bm{x}) = [\nabla \phi_1(\bm{x})^\text{T}, \ldots, \nabla \phi_n(\bm{x})^\text{T}]^\text{T}= (I_{nd} - \mathbf{W})\bm{x}
$. Furthermore, the Riemannian gradient of $\phi(\bm{x})$  satisfies the following result.
\begin{lemma}[\cite{Deng2025}] \label{l3.5}
	For any $\bm{x} \in \mathcal{M}^n$, it holds that
	\begin{align}\label{phi bound}
		\left\| \sum_{i=1}^{n} \operatorname{grad} \phi_i(\bm{x}) \right\| \leq 2\sqrt{n} L_P \| \bm{x} - \bar{\bm{x}} \|^2,
	\end{align}
\end{lemma}
where $\phi(\bm{x})= \frac{1}{4}\sum_{i=1}^{n} \sum_{j=1}^{n} W_{ij} \|x_i-x_j\|^2$.

Lemmas \ref{l3} and \ref{l3.5} facilitate the consensus analysis of PR-EXTRA with respect to the manifold mean, establishing the convergence of each node to a stationary point requires the boundedness of the sequence generated by PR-EXTRA.
\begin{lemma}\label{l4}
	Suppose Assumptions \ref{ass2} and \ref{ass3} hold. Then, for any $i \in \mathcal{V}$ , the norm of the $\eta_{i,k}$ is bounded by
	\begin{equation*}
		\|\eta_{i,k}\| \le 2\tau L_r.
	\end{equation*}
\end{lemma}
\begin{proof}
	See Appendix \ref{Appendix A}.
\end{proof}

\begin{lemma} \label{l5}
	Suppose Assumptions \ref{ass1}--\ref{ass3} hold. If $\{\bm{x}_{k}\}$ is the sequence generated by Algorithm \ref{ass1}, it holds that
	\begin{equation}\label{f desent}
		\begin{aligned}
			&	f(\bar{x}_{k+1}) \\
			&	\leq  f(\bar{x}_k) + \langle \hat{g}_k, \hat{x}_{k+1} - \hat{x}_k \rangle +\frac{L}{n} \|\bm{x}_k - \bar{\bm{x}}_k\|^2  \\
			& \quad +\frac{1}{\alpha^2 L} \left( \| \bar{x}_{k+1} - \hat{x}_{k+1} \|^2 + \| \hat{x}_k - \bar{x}_k \|^2 \right) + \frac{\alpha^2 L}{2} \| \hat{g}_k \|^2.
		\end{aligned}
	\end{equation}
\end{lemma}
\begin{proof}
	See Appendix \ref{Appendix B}.
\end{proof}	 

Lemma \ref{l5} establishes a sufficient descent property for the smooth component of the problem \eqref{composite}. This inequality plays a crucial role in the convergence analysis, and similar analytical techniques have been adopted in prior works on Riemannian distributed algorithms, such as~\cite{REXTRA,Deng2025,Wang2025}.

\begin{lemma} \label{A2}
	Suppose Assumptions \ref{ass1}--\ref{ass3} hold. Then, the following inequality holds:
	\begin{equation*}
		\begin{aligned}
			&\langle \hat{g}_k, \hat{x}_{k+1} - \hat{x}_k \rangle \\
			&\leq  \frac{8L Q + 9L}{n}\left\| \bm{x}_k - \bar{\bm{x}}_k \right\|^2 + \frac{2LQ+2L}{n} \left\| \bm{s}_k\right\|^2 \\
			& \quad -\alpha\|\hat{g}_k\|^2 + (4Q\tau^2 L_r^2+2\tau L_r)\left\| \hat{g}_k \right\|.
		\end{aligned}
	\end{equation*}
\end{lemma}
\begin{proof}
	See Appendix \ref{Appendix D1}.
\end{proof}	
\begin{lemma}\label{l6}
	Suppose Assumptions \ref{ass1}--\ref{ass3} hold. If the initial point satisfies $(\bm{x}_0, -\alpha \mathrm{grad} f(\bm{x}_0)) \in \mathcal{N}(\delta)$ and the stepsizes satisfy $\tau \leq \text{min}\; \left(\frac{\delta}{L_r}, \frac{(1-\bar{\nu})\delta}{12\sqrt{n} L_r}, \sqrt{\frac{(1-\bar{\nu})\delta}{12 Q n L_r^2}}\right)$ and $\alpha \le \text{min}\; \left(\frac{\delta}{L_g}, \frac{(1-\bar{v})\delta}{6\sqrt{5n}L_g}\right)$, then every iterate remains in the same neighbourhood, i.e., $(\bm{x}_{k+1}, \bm{s}_{k+1}) \in \mathcal{N}(\delta)$, 
	where $\delta < \text{min}\; \left(\frac{2-\sqrt{2}}{6}, \frac{1-v}{12}\right)$, $v < 1$, and $\bar{v} = v + 12\delta$.
\end{lemma}
\begin{proof}
	See Appendix \ref{Appendix C}.
\end{proof}	

\begin{lemma}\label{l7}
	Suppose Assumptions \ref{ass1}--\ref{ass3} hold. If the stepsize $\alpha < \frac{1}{8L\sqrt{\tilde{C}_1}}$, then there exist constants $C_0, C_1, C_2, C_3$ given by \eqref{constant2}, such that
	\begin{equation*}\label{x,s bound}
		\begin{aligned}
			&\sum_{k=1}^K \left( \|\bm{x}_k - \bar{\bm{x}}_k\|^2 + \|\mathbf{s}_k\|^2 \right) \\
			&\le \alpha^2 C_1 \sum_{k=1}^K \|\hat{\bm{g}}_k\|^2 + C_2 \sum_{k=1}^K \|\bm{\eta}_k\|^2 + C_3 \sum_{k=0}^K \|\bm{\eta}_k\|^4 + C_0.
		\end{aligned}
	\end{equation*}
	Moreover, there exists a constant $C > 0$ independent of $n, L, \alpha, \tau$ such that
	\begin{equation}\label{x,s bound2}
		\frac{1}{n}(\|\bm{x}_{k} - \bar{\bm{x}}_{k}\|^2+\|\bm{s}_{k} + \alpha \hat{\bm{g}}_{k} \|^2)\leq C(L^2\alpha^2+L_r^2\tau^2+L_r^4\tau^4),
	\end{equation}
	where $\tilde{C}_1 = \frac{4}{(1-\bar{v})^2}$.
\end{lemma}

\begin{proof}
	See Appendix \ref{Appendix D}.
\end{proof}	

Lemmas \ref{l4}--\ref{l7} collectively refine the sufficient descent inequality established in Lemma \ref{l5}. To complete the convergence analysis, it remains to derive a comparable estimate for the nonsmooth term $r$ of the problem \eqref{composite}.

\begin{lemma}\label{l8}
	Suppose Assumptions \ref{ass1}--\ref{ass3} hold. Then, for any iteration $k \in \mathbb{N}$, the function $r(\cdot)$ satisfies
	\begin{equation}\label{r desent}
		\begin{aligned}
			&r(\hat{x}_{k+1}) \\
			&\leq  r(\bar{x}_k) + \left( \frac{1}{2n\tau} + \frac{L_r M}{n} + \frac{8L_r Q}{n} \right) \|\bar{\bm{x}}_k - \bm{x}_k\|^2  \\
			& \quad + \frac{3L_r}{\sqrt{n}} \|\bar{\bm{x}}_k - \bm{x}_k\| \|\bm{s}_k\|^2 +\frac{Q}{n} \|\bm{\eta}_{k}\|^2- \frac{1}{2n\tau} \|\bm{\eta}_k\|^2 \\
			&\quad + \frac{2L_r}{n}+ \frac{L_r+1}{\sqrt{n}} \|\bm{\eta}_k\|.
		\end{aligned}
	\end{equation}
\end{lemma}
\begin{proof}
	See Appendix \ref{Appendix E}.
\end{proof}

Lemma \ref{l8} establishes a descent property with respect to the nonsmooth term \( r \). 
By combining this result with the smooth-part descent established in Lemma \ref{l5}, 
we obtain a global descent inequality for the overall function \( h(x) \), 
leading to the following convergence theorem.
\begin{theorem}\label{t1}
	Suppose Assumptions \ref{ass1}--\ref{ass3} hold. If the there exist stepsizes satisfy
	\begin{subequations}
		\begin{align}
			\alpha &\le \text{min}\;\left\{ \frac{1}{8L\sqrt{C_1}}, \frac{1}{4 C_1(D_1+D_2+D_3)} \right\} ,\\
			\tau &\le \text{min}\;\left\{ \frac{1}{16+4nC_2}, \left(\frac{1}{32nL_r^2C_3}\right)^{1/3} \right\},
		\end{align}
	\end{subequations}
	where $C_i,D_i, i=1,2,3$ are constants given by \eqref{constant2} and \eqref{constant}, then the sequence $\{\bm{x}_k\}$ generated by Algorithm \ref{algorithm_extra} satisfies the following complexity bound
	\begin{equation*}
		\underset{0\le k\le K}{\text{min}} 	\text{max} \{\|\mathrm{grad} f(\bar{x}_k)\|^2, \|\bm{x}_k - \bar{\bm{x}}_k\|^2, \|\bm{\eta}_k\|^2\} \le \mathcal{O}\left(\frac{1}{K}\right).
	\end{equation*}
\end{theorem}
\begin{proof}
	See Appendix \ref{Appendix F}.
\end{proof}	

\begin{theorem}\label{t2}
	Suppose Assumptions \ref{ass1}--\ref{ass3} hold. Then for the sequence $\{(\bm{x}_k, \bm{y}_k, \bm{\eta}_k)\}$ generated by Algorithm \ref{algorithm_extra}, the following results hold,
	\begin{enumerate}
		\item For any accumulation point $\bm{x}^\ast$ of $\{\bm{x}_k\}$, there exists $\bar{x}^\ast \in \mathcal{M}$ such that $\bm{x}^\ast = (1_n \otimes I_d) \bar{x}^\ast$.
		\item The auxiliary sequence $\{\bm{y}_k\}$ shares the same accumulation points as $\{\bm{x}_k\}$. Specifically, for any accumulation point $\bm{x}^\ast$, there exists an accumulation point $\bm{y}^\ast$ of $\{\bm{y}_k\}$ such that $\bm{y}^\ast = (1_n \otimes I_d) \bar{x}^\ast$.
		\item The point $\bar{x}^\ast$ is a stationary point of problem \eqref{composite}, i.e.,
		\begin{equation*}
			0 \in \mathcal{P}_{T_{\bar{x}^\ast}\mathcal{M}} \bigl(\nabla f(\bar{x}^\ast) + \partial r(\bar{x}^\ast)\bigr).
		\end{equation*}
	\end{enumerate}
\end{theorem}
\begin{proof}
	See Appendix \ref{Appendix G}.
\end{proof}	

\begin{remark}
	Theorem \ref{t1} shows that PR-EXTRA achieves an oracle-type convergence rate of \(\mathcal{O}(1/K)\) in terms of a composite measure that simultaneously captures gradient norm, consensus error, and optimality gap. This rate matches the best-known complexity lower bound for distributed first-order algorithms in Euclidean composite optimization \cite{Shi-2015}. Theorem \ref{t2} further strengthens this result by establishing the subsequential convergence of the iterates to a consensus point that is a Riemannian stationary point of the composite problem.
\end{remark}

\section{Numerical Experiments}\label{sec5}
In this section, we compare our proposed PR-EXTRA with DR-ProxGT \cite{Wang2025}, DRSM \cite{Deng2024}, in distributed principal component analysis problem \eqref{SPCA} and coordinate-independent
sparse estimation problem \eqref{CISE}. 

The experiments are conducted over a random network generated by the Erdős-Rényi model with a connection probability of $p=0.6$, yielding an average degree of approximately 4.8. To ensure the doubly stochastic condition required for consensus, the weight matrix $W$ is configured using Metropolis-Hastings weights, defined as $w_{ij} = (\max(|\mathcal{N}_i|,|\mathcal{N}_j|)+1)^{-1}$ for neighbor pairs $(i,j) \in \mathcal{E}$, with diagonal elements adjusted to preserve row stochasticity.

For the algorithmic implementation, PR-EXTRA employs a stepsize $\alpha=0.001$ and a proximal gradient stepsize $\tau=0.001$. DR-ProxGT is configured with a stepsize $\alpha=1$, and proximal gradient stepsize $\tau=0.0001$ per iteration. DRSM uses a diminishing stepsize $(k+1)^{-1/2}$.


We evaluate algorithm performance using the following metrics,
\begin{itemize}
	\item KKT violation: $\|\mathcal{P}_{T_{\bar x}\mathcal{M}}(\nabla f(\bar x) + \lambda \partial \|\bar x\|)\|$.
	\item Consensus error: $\frac{1}{n}\sum_{i=1}^n \|x_i - \bar{x}\|^2$.
\end{itemize}

The optimization terminates when either: (1) maximum iterations $K_{\max}=3000$ is reached, or (2) consensus error falls below $\epsilon_{\text{cons}}=10^{-12}$.

\subsection{SPCA Problem}
We evaluate the proposed algorithm on the distributed sparse principal component analysis (SPCA) problem. By incorporating an $\ell_1$-regularization term into the standard PCA framework, the problem seeks loading vectors that balance variance maximization with sparsity. The composite optimization problem is solved over a communication network generated by the Erdős-Rényi graph
\begin{align}\label{SPCA}
	\underset{x \in \mathcal{M}}{\text{min}}  \;  \frac{1}{n} \sum_{i=1}^{n} \left( -\frac{1}{2} \operatorname{tr}\left( x^{\text{T}} A_i^{\text{T}} A_i x \right) + \lambda \| x\|_1 \right),
\end{align}
where the decision manifold is the product of Stiefel manifolds $\mathcal{M}:=\operatorname{St}(d,r)$, $\|x\|_1$:=$\sum_{i=1}^{d}\sum_{j=1}^{r}|x(i,j)|$, with $x(i,j)$ being the $i,j$-th element of $x$. . The $A_{i}\in\mathbb{R}^{m_{i}\times d}$ denotes the local data matrix residing at node $i$.

To simulate a controlled numerical environment, we synthesize a global data matrix $B \in \mathbb{R}^{m \times d}$ ($m=8000, d=10$) via singular value decomposition $B=U\Sigma V^{\text{T}}$. The spectral properties are governed by a geometric progression of singular values $\tilde{\Sigma}=\operatorname{diag}(\xi^{j})$ for $j=0,\dots,d-1$, where $\xi=0.8$ induces a significant eigengap. The final data matrix $A=U\tilde{\Sigma}V^{\text{T}}$ is uniformly partitioned row-wise into $n=8$ disjoint subsets $\{A_i\}_{i=1}^{n}$, ensuring each node processes $m_i=1000$ samples while preserving the global statistical profile. We set the target dimension to $r=5$ and the regularization parameter to $\lambda = 0.001$ to promote meaningful sparsity. As illustrated in Figure \ref{fig:pca_comparison_combined}, PR-EXTRA exhibits the rapid reduction in KKT violation among all tested algorithms. Both its stationarity (KKT) violation and consensus error stabilize within approximately $1000$ iterations, outperforming DR-ProxGT, which requires nearly $3000$ iterations to reach a comparable steady state despite its aggressive initial consensus convergence.

\begin{figure}[htbp]
	\centering
	\includegraphics[width=0.47\textwidth]{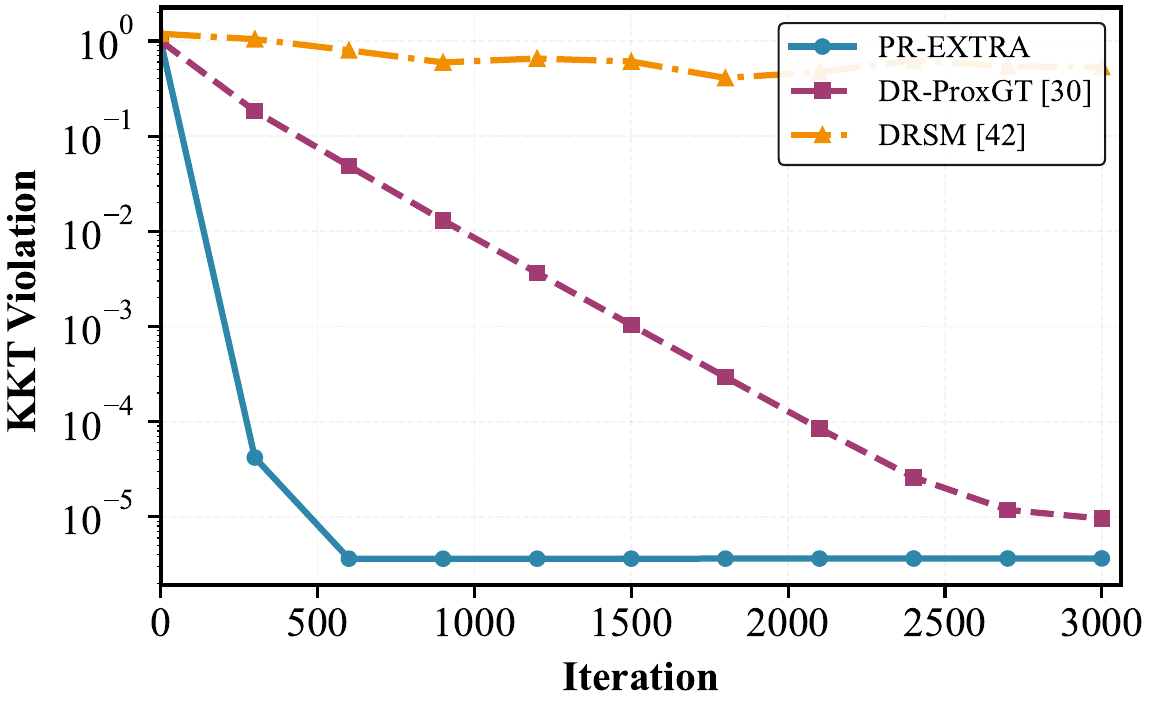}
	\hfill
	\includegraphics[width=0.47\textwidth]{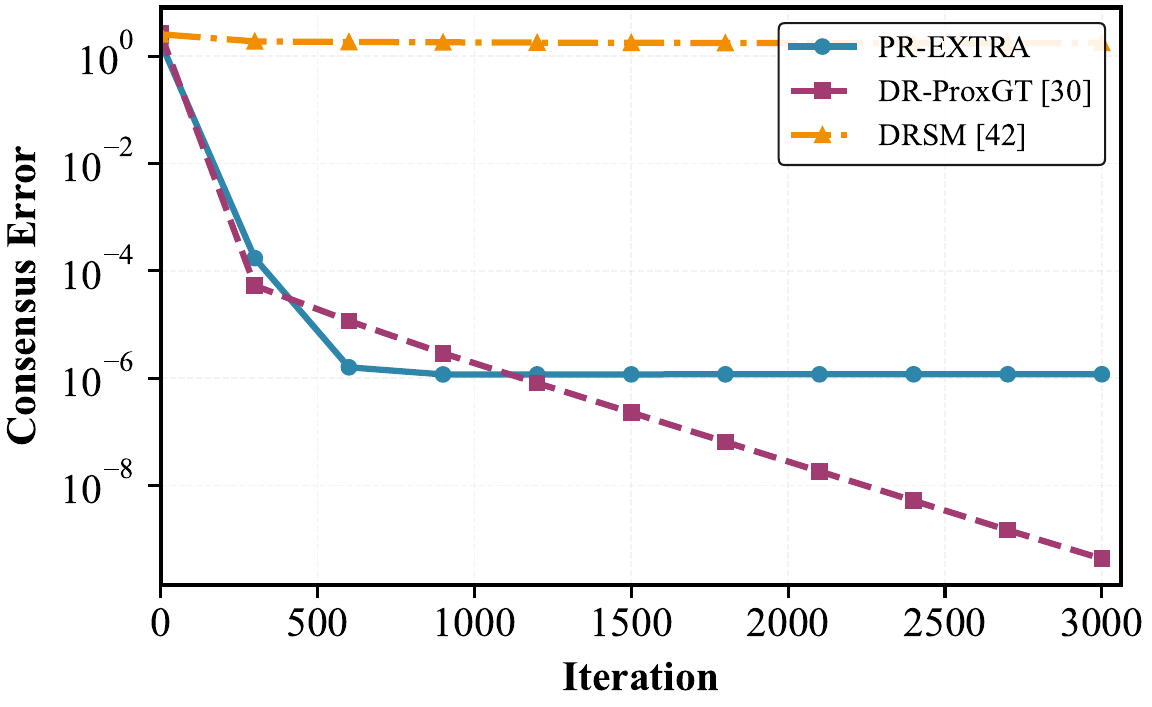}
	\caption{Numerical comparison of DR-ProxGT, DRSM, and PR-EXTRA on the SPCA Problem. The up and down figures depict stationarity violations and consensus errors, respectively.}
	\label{fig:pca_comparison_combined}
\end{figure}

\subsection{CISE Problem}
The distributed sparse invariant subspace extraction (CISE) problem \cite{CISE} extends invariant subspace computation by employing an $\ell_{2,1}$-regularizer to induce row-wise sparsity in the basis. This formulation facilitates the identification of dominant features across all nodes within the network generated by the Erdős-Rényi graph
\begin{equation}\label{CISE}
	\underset{x \in \mathcal{M}}{\text{min}} \;\frac{1}{n} \sum_{i=1}^{n} \left( -\frac{1}{2} \operatorname{tr}\left( x^{\text{T}} A_i^{\text{T}} A_i x \right) + \lambda \| x \|_{2,1} \right),
\end{equation}
where $\| x\|_{2,1}:=\sum_{i=1}^d{\| x(i)\|}$. The manifold structure and node configuration remain consistent with the SPCA setup.

The synthetic data for CISE are generated with a spectral decay $\tilde{\Sigma}= \operatorname{diag}(\xi^{j/2})$ for $j=0,\dots,d-1$ and $\xi=0.8$, providing a more nuanced eigenvalue distribution. The global matrix is similarly partitioned among $n=8$ nodes ($m_i=1000$). We set $\lambda = 0.01$ and $r=5$ to ensure subspace fidelity while enforcing row-wise structural sparsity. Consistent with the SPCA results, PR-EXTRA shows superior convergence characteristics. The stationarity violation and consensus error for PR-EXTRA both converge after approximately $1800$ iterations, reaffirming its efficiency in handling structured non-smooth regularizers on manifolds.

\begin{figure}[htbp]\label{fig2}
	\centering
	\includegraphics[width=0.47\textwidth]{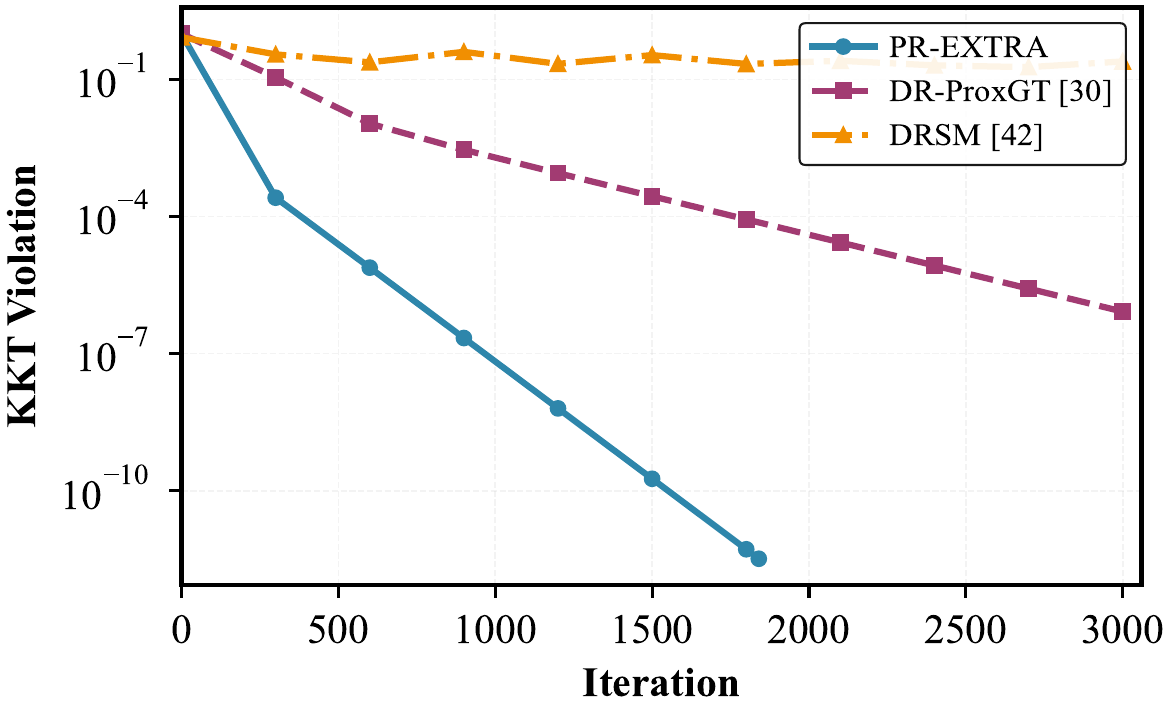}
	\hfill
	\includegraphics[width=0.47\textwidth]{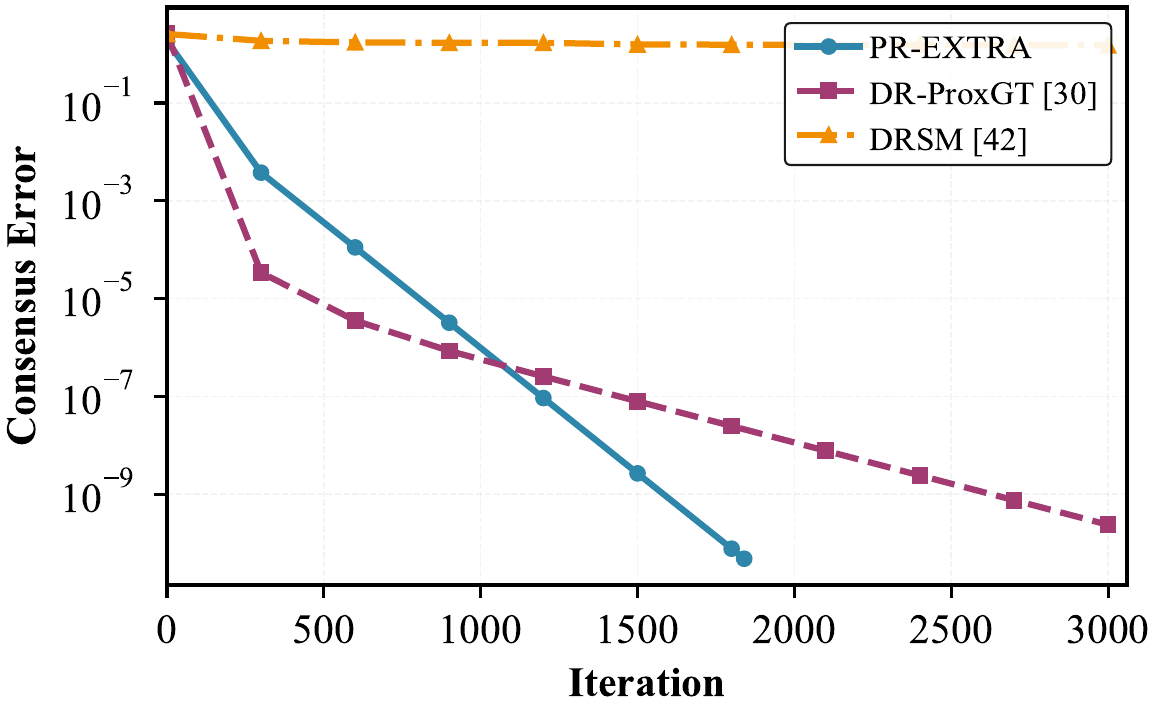}
	\caption{Numerical comparison of DR-ProxGT, DRSM, and PR-EXTRA on the CISE Problem. The up and down figures depict stationarity violations and consensus errors, respectively.}
	\label{fig:cise_comparison_combined}
\end{figure}

\section{Conclusion}\label{sec6}
In this paper, we addresse the challenges inherent in distributed Riemannian composite optimization—a problem complicated by the intrinsic nonsmoothness and nonconvexity of the objective function. Existing algorithms typically resort to multi-step consensus loops to ensure feasibility, often falling short in terms of communication efficiency and computational costs. To overcome these limitations, we propose a novel loopless distributed Riemannian proximal gradient EXTRA algorithm for composite optimization over the Riemannian manifold. We rigorously show that our algorithm achieves exact convergence to a stationary point under standard assumptions. Furthermore, we established a convergence rate of $\mathcal{O}(1/K)$, a notable result that matches the best-known rates for distributed nonconvex optimization while significantly reducing communication overhead \cite{REXTRA}. Numerical experiments show the effectiveness of the proposed algorithm. While the potential of our approach is evident, several intriguing avenues remain for exploration. For instance, extending this framework to the distributed stochastic optimization setting on Riemannian manifolds. Additionally, investigating its generalization to asynchronous scenarios could further enhance practical efficiency in heterogeneous networks.

	\appendix
\section{Proof of Lemma \ref{l4}}\label{Appendix A}

To begin with, the assertion holds trivially if $\eta_{i,k} = 0$. Next, we consider the case where $\eta_{i,k} \neq 0$. For convenience, we define the objective function as
\[
g_{i,{k}}(\eta_{i,k}) := \frac{1}{2\tau} \|\eta_{i,k}\|^2 + r(y_{i,k} + \eta_{i,k}).
\]
Since $g_{i,{k}}$ is strongly convex with modulus $1/\tau$, we have
\begin{align}\label{g convex}
	g_{i,{k}}(\eta'_{i,k}) &\ge g_{i,{k}}(\eta_{i,k}) + \langle \partial g_{i,{k}}(\eta_{i,k}), \eta'_{i,k} - \eta_{i,k} \rangle \notag \\
	&\quad + \frac{1}{2\tau} \|\eta'_{i,k} - \eta_{i,k}\|^2, 
\end{align}
for any $\eta_{i,k}, \eta'_{i,k} \in \mathbb{R}^{n \times p}$. In particular, for $\eta_{i,k}, \eta'_{i,k} \in T_{y_{i,k}}\mathcal{M}$, it holds that
\[
\langle \partial g_{i,{k}}(\eta_{i,k}), \eta'_{i,k} - \eta_{i,k} \rangle = \langle \mathcal{P}_{T_{y_{i,k}}\mathcal{M}}(\partial g_{i,{k}}(\eta_{i,k})), \eta'_{i,k} - \eta_{i,k} \rangle.
\]
The first-order optimality condition implies that
\begin{equation*}
	0 \in \mathcal{P}_{T_{y_{i,k}}\mathcal{M}}(\partial g_{i,{k}}(\eta_{i,k})).
\end{equation*}
By setting $\eta'_{i,k} = 0$ in \eqref{g convex} and utilizing the optimality condition, we obtain
\[
g_{i,{k}}(0) - g_{i,{k}}(\eta_{i,k}) \ge \frac{1}{2\tau} \|\eta_{i,k}\|^2.
\]
This, combined with the Lipschitz continuity of $r$, implies that
\[
\frac {1}{2\tau} \|\eta_{i,k}\|^2 \le r(x_{i,k}) - r(x_{i,k}+ \eta_{i,k}) \le L_r \|\eta_{i,k}\|.
\]
Consequently, we conclude that $\|\eta_{i,k}\| \le 2 \tau L_r$, 
which completes the proof.\quad \qed

\section{Proof of Lemma \ref{l5}}\label{Appendix B}
The objective is to derive a refined descent estimate for the smooth function $f$, which decouples the ideal descent term from the consensus errors and the gradient norm. Invoking the $L$-smoothness of $f$, we start with the standard inequality
\begin{align*}
	&	f(\bar{x}_{k+1})\\
	&\le f(\bar{x}_{k})
	+ \langle \operatorname{grad} f(\bar{x}_{k}),\, \bar{x}_{k+1} - \bar{x}_{k} \rangle
	+ \frac{L}{2}\|\bar{x}_{k+1} - \bar{x}_{k}\|^2 \\
	&\le f(\bar{x}_{k})
	+ \langle \hat{g}_k,\, \bar{x}_{k+1} - \bar{x}_k \rangle + \frac{L}{2}\|\bar{x}_{k+1} - \bar{x}_{k}\|^2
	\\
	& \quad + \langle \operatorname{grad} f(\bar{x}_{k}) - \hat{g}_k,\, \bar{x}_{k+1} - \bar{x}_{k} \rangle  \\
	&\leq f(\bar{x}_k) + \langle \hat{g}_k, \bar{x}_{k+1} - \bar{x}_k \rangle + \frac{3L}{4} \|\bar{x}_{k+1} - \bar{x}_k\|^2 \\
	&\quad + \frac{1}{L} \|\operatorname{grad}  f(\bar{x}_k) - \hat{g}_k\|^2 \\
	&\leq f(\bar{x}_k) + \langle \hat{g}_k, \hat{x}_{k+1} - \hat{x}_k \rangle +\frac{L}{n} \|\bm{x}_k - \bar{\bm{x}}_k\|^2  \\
	&\quad + .\langle \hat{g}_k, \bar{x}_{k+1} - \hat{x}_{k+1} + \hat{x}_k - \bar{x}_k \rangle \\
	&\quad  + \frac{3L}{4} \|\bar{x}_{k+1} - \bar{x}_k\|^2 \\
	& \leq f(\bar{x}_k) + \langle \hat{g}_k, \hat{x}_{k+1} - \hat{x}_k \rangle +\frac{L}{n} \|\bm{x}_k - \bar{\bm{x}}_k\|^2  \\
	& \quad  +\frac{1}{\alpha^2 L} \left( \| \bar{x}_{k+1} - \hat{x}_{k+1} \|^2 + \| \hat{x}_k - \bar{x}_k \|^2 \right)\\
	& \quad + \frac{\alpha^2 L}{2} \| \hat{g}_k \|^2 ,
\end{align*}
where the fourth inequality is by $
\|\hat{g}_k - \operatorname{grad}  f(\bar{x}_k)\|^2
\le \frac{1}{n}\sum_{k=1}^n \|\operatorname{grad}  f_i(x_k) - \operatorname{grad}  f_i(\bar{x}_k)\|^2
\le \frac{L^2}{n}\|\bm{x}_k - \bar{\bm{x}}_k\|^2 
$. Thus, we obtain the descent inequality \eqref{f desent} for the function $f$. \quad \qed
\section{Proof of Lemma \ref{A2}}\label{Appendix D1}
This lemma focuses on bounding the inner product term $\langle \hat{g}_k, \hat{x}_{k+1} - \hat{x}_k \rangle$ appearing in Lemma~\ref{l5}. 
Recalling the update rule of PR-EXTRA, we decompose this term as follows
\begin{align*}
	&\left\langle \hat{g}_k, \hat{x}_{k+1} - \hat{x}_k \right\rangle \\
	&= \left\langle \hat{g}_k, \frac{1}{n} \sum_{i=1}^n \left( x_{i,k+1} - x_{i,k} - s_{i,k} + \nabla \phi_i(\bm{x}_k) \right) \right\rangle \\
	&\quad +\left\langle \hat{g}_k, \frac{1}{n} \sum_{i=1}^n \left( s_{i,k} - \nabla \phi_i(\bm{x}_k) \right) \right\rangle\\
	& = \left\langle \hat{g}_k, \frac{1}{n} \sum_{i=1}^n \left( x_{i,k+1} - x_{i,k} - s_{i,k} + \nabla \phi_i(\bm{x}_k) \right) \right\rangle -\alpha \|\hat{g}_k\|^2.
\end{align*}
By introducing the variable $y_{i,k}$, we can further split the first term
\begin{align}\label{eq:inner_prod_split}
	&\left\langle \hat{g}_k, \frac{1}{n} \sum_{i=1}^n \left( x_{i,k+1} - x_{i,k} - s_{i,k} + \nabla \phi_i(\bm{x}_k) \right) \right\rangle \notag \\
	&= \left\langle \hat{g}_k, \frac{1}{n} \sum_{i=1}^n \left( x_{i,k+1} - y_{i,k} + y_{i,k} - x_{i,k} - s_{i,k} + \nabla \phi_i(\bm{x}_k) \right) \right\rangle  \notag\\
	&= \left\langle \hat{g}_k, \frac{1}{n} \sum_{i=1}^n \left( x_{i,k+1} - y_{i,k} \right) \right\rangle  \notag\\
	&\quad + \left\langle \hat{g}_k, \frac{1}{n} \sum_{i=1}^n \left( y_{i,k} - x_{i,k} - s_{i,k} + \nabla \phi_i(\bm{x}_k) \right) \right\rangle.
\end{align}

We now bound the two terms on the right-hand side of \eqref{eq:inner_prod_split} separately. Define $d_i = \nabla \phi_i(\bm{x}_k) - s_{i,k}$, and decompose it orthogonally into its tangential and normal components with respect to the manifold at $\bm{x}_{i,k}$:
\begin{align*}
	d_{i,1} = \mathcal{P}_{T_{\bm{x}_{i,k}} \mathcal{M}}(d_i), \qquad d_{i,2} = d_i - d_{i,1}.
\end{align*}
Due to orthogonality, $\|d_i\|^2 = \|d_{i,1}\|^2 + \|d_{i,2}\|^2$. 
The second term in \eqref{eq:inner_prod_split} can be bounded by
\begin{align}\label{eq:inner_prod_split2}
	&\left\langle \hat{g}_k, \frac{1}{n} \sum_{i=1}^n \left( y_{i,k} - x_{i,k} + d_i \right) \right\rangle \notag\\
	&= \frac{1}{n} \sum_{i=1}^n \left\langle \hat{g}_k, \mathcal{P}_\mathcal{M}(x_{i,k} - d_{i,1} - d_{i,2}) - (x_{i,k} - d_{i,1}) \right\rangle \notag\\
	&\quad + \frac{1}{n} \sum_{i=1}^n \left\langle \hat{g}_k, d_{i,2} \right\rangle \notag\\
	&\leq \frac{L Q}{n} \sum_{i=1}^n \left\| d_i \right\|^2 + \frac{1}{n} \sum_{i=1}^n \left\langle \hat{g}_k, d_{i,2} \right\rangle.
\end{align}
Crucially, since $\operatorname{grad}f_i(x_{i,k}) \in T_{x_{i,k}}\mathcal{M}$ and $d_{i,2}$ is in the normal space, their inner product vanishes. Thus, we can rewrite the last term of \eqref{eq:inner_prod_split2} as $
\langle \hat{g}_k, d_{i,2} \rangle = \langle \hat{g}_k - \operatorname{grad}f_i(x_{i,k}), d_{i,2} \rangle.
$
Applying Young's inequality, we obtain
\begin{align*}
	&\left\langle \hat{g}_k, \frac{1}{n} \sum_{i=1}^n \left( y_{i,k} - x_{i,k} + d_i \right) \right\rangle \\
	&\leq \frac{L Q}{n} \sum_{i=1}^n \left\| d_i \right\|^2 + \frac{1}{4nL} \sum_{i=1}^n \| \hat{g}_k - \operatorname{grad}f_i(x_{i,k}) \|^2 \\
	&\quad + \frac{L}{n} \sum_{i=1}^n \|d_{i,2}\|^2\\
	&\leq \frac{L Q + L}{n} \left\| (I-\mathbf{W}) \bm{x}_k + \bm{s}_k \right\|^2 + \frac{L}{n} \left\| \bm{x}_k - \bar{\bm{x}}_k \right\|^2,\\
	&\leq  \frac{8L Q + 9L}{n}\left\| \bm{x}_k - \bar{\bm{x}}_k \right\|^2 + \frac{2LQ+2L}{n} \left\| \bm{s}_k\right\|^2.
\end{align*}
Recall that $x_{i,k+1} = \mathcal{P}_{\mathcal{M}}(y_{i,k} + \eta_{i,k})$ with $\eta_{i,k} \in T_{y_{i,k}}\mathcal{M}$.
Using Lemma \ref{L1}, we conclude that
\begin{align*}
	&\left\langle \hat{g}_k, \frac{1}{n} \sum_{i=1}^n \left( x_{i,k+1} - y_{i,k} \right) \right\rangle \\
	&= \frac{1}{n} \sum_{i=1}^n \left\langle \hat{g}_k, \mathcal{P}_\mathcal{M}(y_{i,k} + \eta_{i,k}) - (y_{i,k} + \eta_{i,k}) \right\rangle + \frac{1}{n} \sum_{i=1}^n \left\langle \hat{g}_k, \eta_{i,k} \right\rangle \\
	&\leq \frac{Q}{n} \sum_{i=1}^n \left\| \eta_{i,k} \right\|^2 \left\| \hat{g}_k \right\| + \frac{1}{n} \sum_{i=1}^n \left\| \eta_{i,k} \right\| \left\| \hat{g}_k \right\| \\
	&\leq (4Q\tau^2 L_r^2+2\tau L_r)\left\| \hat{g}_k \right\|, 
\end{align*}
which completes the proof. \quad \qed

\section{Proof of Lemma \ref{l6}}\label{Appendix C}
We establish the invariance of the neighborhood $\mathcal{N}(\delta)$ via mathematical induction. Suppose that $(\bm{x}_{k}, \bm{s}_{k}) \in \mathcal{N}(\delta)$ holds for some $k \ge 0$. 
Since $y_{i,k} \in \mathcal{M}$ and the $\bm{\eta}_{k}$ is controlled by the stepsize $\tau$, we can select a sufficiently small $\tau$ such that $\|\eta_{i,k}\| \leq \delta$. 
Additionally, the convex combination of $x_{j,k}$ satisfies 
\[
\Big\| \sum_{j=1}^n w_{ij}x_{j,k} - \bar{x}_k \Big\| \leq \sum_{j=1}^n w_{ij} \|x_{j,k} - \bar{x}_k\| \leq \delta,
\]
which ensures that the arguments of the projection remain within the domain where the retraction properties hold.

Firstly, we derive the bound for the primal consensus error. Invoking the optimality of the Riemannian mean $\bar{\bm{x}}_{k+1}$ and the property from Lemma~\ref{L1}, we have
\begin{equation}
	\begin{aligned}
		&\|\bm{x}_{k+1} - \bar{\bm{x}}_{k+1}\| \\
		&\le \|\bm{x}_{k+1} - \bar{\bm{x}}_k\| \\
		&= \|\mathcal{P}_{\mathcal{M}^n}(\bm{y}_{k} + \bm{\eta}_k) - \bar{\bm{x}}_k-\bm{y}_{k} - \bm{\eta}_k+\bm{y}_{k} + \bm{\eta}_k\| \\
		&\le Q \|\bm{\eta}_k \|^2 +\|\bm{y}_{k} + \bm{\eta}_k - \bar{\bm{x}}_k \|  \\
		&\le  Q \|\bm{\eta}_k \|^2+  \|\mathcal{P}_{\mathcal{M}^n}(\mathbf{W}\bm{x}_k + \bm{s}_k) + \bm{\eta}_k - \mathcal{P}_{\mathcal{M}^n}(\hat{\bm{x}}_k)\| \\
		&\le  Q \|\bm{\eta}_k \|^2+ \frac{1}{1-3\delta} \|\mathbf{W}\bm{x}_k + \bm{s}_k + \alpha\hat{\bm{g}}_k - \alpha\hat{\bm{g}}_k - \hat{\bm{x}}_k\| \\
		&\quad +  \frac{1}{1-3\delta} \|\bm{\eta}_k \| . \notag
	\end{aligned}
\end{equation}
where the first inequality is from the optimality of $\bar{\bm{x}}_{k+1}$, the second inequality follows from Lemma~\ref{L1}, equation \eqref{L1-eq}. The last inequality holds due to the bounds \( \| s_{i,k} + \alpha \hat{g}_k \| \leq \delta \), \( \alpha \| \hat{g}_k \| \leq \delta \), and
\(	
\Big\| \sum_{j=1}^{n} \frac{1}{n} x_{j,k} - \bar{x}_k \Big\| \leq \sum_{j=1}^{n} \frac{1}{n} \| x_{j,k} - \bar{x}_k \| \leq \delta,
\)
together with the \( 1/(1-3\delta) \)-Lipschitz continuity of \( \mathcal{P}_{\mathcal{M}} \) as given in \eqref{R-proximal}. Secondly, we bounded the 	$\|\bm{s}_{k+1} - \hat{\bm{s}}_{k+1}\| $,
\begin{align*}
	&\|\bm{s}_{k+1} - \hat{\bm{s}}_{k+1}\| \\
	&= \|\bm{s}_{k+1} + \alpha\hat{\bm{g}}_{k+1}\| \\
	&= \|(\mathbf{W} - \tilde{\mathbf{W}})\bm{x}_k + \bm{s}_k - \alpha(\text{grad} f(\bm{x}_{k+1}) - \text{grad} f(\bm{x}_k)) \\
	&\quad + \alpha(\hat{\bm{g}}_{k+1} - \hat{\bm{g}}_k) + \alpha\hat{\bm{g}}_k\| \\
	&= \|(\mathbf{W} - \tilde{\mathbf{W}})(\bm{x}_k - \hat{\bm{x}}_k) - \alpha(\text{grad} f(\bm{x}_{k+1}) - \text{grad} f(\bm{x}_k)) \\
	&\quad + \bm{s}_k + \alpha\hat{\bm{g}}_k  + \alpha(\hat{\bm{g}}_{k+1} - \hat{\bm{g}}_k)\|,
\end{align*}
where the first equality follows from the relation $\hat{\bm{s}}_{k+1} = -\alpha \hat{\bm{g}}_{k+1}$, while the last equality utilizes the property $(\mathbf{W} - \tilde{\mathbf{W}})\hat{\bm{x}}_k = \bm{0}$. Recalling that $J(\operatorname{grad} f(\bm{x}_{k+1}) - \operatorname{grad} f(\bm{x}_k)) = \hat{\bm{g}}_{k+1} - \hat{\bm{g}}_{k}$, we combine the preceding bounds to establish the following matrix-form inequality:
\begin{equation}\label{local bound}
	\begin{aligned}[b]
		&\left\| \begin{pmatrix} 
			\bm{x}_{k+1} - \bar{\bm{x}}_{k+1} \\ 
			\bm{s}_{k+1} + \alpha \hat{\bm{g}}_{k+1} 
		\end{pmatrix} \right\| \\
		&\le \left\| \begin{pmatrix} 
			\frac{1}{1-3\delta} \mathbf{W} - J & \frac{1}{1-3\delta} I \\ 
			\mathbf{W} - \tilde{\mathbf{W}} & I - J
		\end{pmatrix} \begin{pmatrix} 
			\bm{x}_k - \hat{\bm{x}}_k \\ 
			\bm{s}_k + \alpha \hat{\bm{g}}_k 
		\end{pmatrix} \right\| \\
		&\quad + 2\alpha \left\| \begin{pmatrix} 
			\hat{\bm{g}}_k \\ 
			(I - J) \Delta \mathbf{g}_{k+1} 
		\end{pmatrix} \right\| 
		+ \frac{2}{1-3\delta} \left\| \begin{pmatrix} 
			\bm{\eta}_k \\ \bm{0} 
		\end{pmatrix} \right\| \\
		&\le \|N\| \left\| \begin{pmatrix} 
			\bm{x}_k - \bar{\bm{x}}_k \\ 
			\bm{s}_k + \alpha \hat{\bm{g}}_k 
		\end{pmatrix} \right\| \\
		&\quad + 2\alpha \left\| \begin{pmatrix} 
			\hat{\bm{g}}_k \\ 
			\Delta \mathbf{g}_{k+1}
		\end{pmatrix} \right\| 
		+ \frac{\|\bm{\eta}_k\|}{1-3\delta}  
		+ Q\|\bm{\eta}_k\|^2,
	\end{aligned}
\end{equation}
where $\Delta \mathbf{g}_{k+1} :=\operatorname{grad} f(\bm{x}_{k+1}) - \operatorname{grad} f(\bm{x}_k)$. To bound the spectral norm of the system matrix $N$, we decompose it as:
\begin{align*}
	&\|N \| \\
	&= \left\| \begin{pmatrix} 
		\frac{1}{(1-3\delta)^2} \mathbf{W} - J & \frac{1}{(1-3\delta)^2} I \\ 
		\mathbf{W} - \tilde{\mathbf{W}} & I - J
	\end{pmatrix} \right\| \\
	&\le \underbrace{\left\|
		\begin{pmatrix} 
			\mathbf{W} - J & I \\ 
			\mathbf{W} - \tilde{\mathbf{W}} & I - J
		\end{pmatrix}
		\right\|}_{v}
	+ \left\|\begin{pmatrix} 
		\frac{3\delta(2-3\delta)}{(1-3\delta)^2}\mathbf{W} &  \frac{3\delta(2-3\delta)}{(1-3\delta)^2}I \\ 
		\bm{0} & \bm{0}
	\end{pmatrix} \right\|.
\end{align*}
According to Lemma 1 in \cite{Pu2025}, the spectral norm of the auxiliary matrix is bounded by unity, i.e., $v \le 1$. Given that $\delta \le \frac{1}{6}$, we have $\frac{1}{1-3\delta} - 1 = \frac{3\delta}{1-3\delta} \le 6\delta$, which implies $\|N\| \leq v + 12\delta =: \bar{v}$. 
Substituting this back, and employing Assumption \ref{ass2}, Lemma \ref{l4}, and the boundedness of $\|\bm{x}_{k} - \bar{\bm{x}}_{k}\|$, we obtain
$
\|(\bm{x}_{k+1} - \bar{\bm{x}}_{k+1}, \bm{s}_{k+1} + \alpha \hat{\bm{g}}_{k+1})\| 
\leq \bar{v}\delta + 2\alpha\sqrt{5n}L_g + \frac{2\tau\sqrt{n}L_r}{1-3\delta} + 4Q\tau^2 n L_r^2 
\leq \delta.
$
This inequality implies that $(\bm{x}_{k+1}, \bm{s}_{k+1}) \in \mathcal{N}(\delta)$.
\quad \qed

\section{Proof of Lemma \ref{l7}}\label{Appendix D}
To streamline the analysis of the squared errors $\|\bm{x}_{k+1} - \bar{\bm{x}}_{k+1}\|^2$ and $ \|\mathbf{s}_{k+1}\|^2$, we define
\begin{align*}
	\mathbf{X}_{k+1}: &= \begin{pmatrix} 
		\bm{x}_{k+1} - \bar{\bm{x}}_{k+1} \\ 
		\mathbf{s}_{k+1} + \alpha \hat{\bm{g}}_{k+1} 
	\end{pmatrix}, \\
	\mathbf{C}_k: &= \begin{pmatrix} 
		\hat{\bm{g}}_k \\ 
		\text{grad} f(\bm{x}_{k+1}) - \text{grad} f(\bm{x}_k) 
	\end{pmatrix},\\
	\mathbf{d}_k :&= \begin{pmatrix} 
		\bm{\eta}_k \\ 
		\bm{0} 
	\end{pmatrix}. 
\end{align*}
From the recursive relationship established in \eqref{local bound}, it follows that
\begin{equation}\label{X_k bound}
	\|\mathbf{X}_{k+1}\| \le \bar{\nu} \|\mathbf{X}_k\| + 2\alpha\|\mathbf{C}_k\| + \frac{1}{1-3\delta}\|\mathbf{d}_k\| + Q\|\mathbf{d}_k\|^2,
\end{equation}
where $\bar{\nu} \in (0,1)$. To derive a bound on the cumulative error, we invoke Lemma 2 of \cite{Xie2015}. This lemma states that for two positive scalar sequences $\{v_k\}_{k \ge 0}$ and $\{w_k\}_{k \ge 0}$ satisfying $v(k+1) \le \eta v(k) + w(k)$ with decaying factor $\eta \in (0,1)$, let $\Upsilon(k) = \sqrt{\sum_{i=0}^k \|v(i)\|^2}$ and $\Omega(k) = \sqrt{\sum_{i=0}^k \|w(i)\|^2}$ be the signal energy from $0$ to $k$. Then, we have $\Upsilon(k) \le \alpha \Omega(k) + \epsilon$, where $\alpha = \frac{\sqrt{2}}{1-\eta}$ and $\epsilon = \sqrt{\frac{2}{1-\eta^2}} v(0)$. 

By applying this result to \eqref{X_k bound} and leveraging the gradient bounds from Lemma \ref{l6}, there exist constants
\begin{align*}
	\tilde{C}_0 &= \frac{5L^2}{1-\bar{\nu}^2}, \quad 
	\tilde{C}_1 = \frac{12}{(1-\bar{\nu})^2}, \\
	\tilde{C}_2 &= \frac{3}{(1-\bar{\nu})^2(1-3\delta)^2}, \quad 
	\tilde{C}_3 = \frac{3Q^2}{(1-\bar{\nu})^2}.
\end{align*}
Summing the squared norms in \eqref{X_k bound} from $k=0$ to $K$ yields
\begin{equation}\label{xie bound}
	\begin{aligned}
		&\sum_{k=0}^K \left( \|\bm{x}_{k+1} - \bar{\bm{x}}_{k+1}\|^2 + \|\mathbf{s}_{k+1} + \alpha \hat{\bm{g}}_{k+1}\|^2 \right)  \\
		&\quad \le \alpha^2 \tilde{C}_1 \sum_{k=0}^K \left( \|\hat{\bm{g}}_k\|^2 + \|\text{grad} f(\bm{x}_{k+1}) - \text{grad} f(\bm{x}_k)\|^2 \right) \\
		&\quad\quad + \tilde{C}_2 \sum_{k=0}^K \|\bm{\eta}_k\|^2  + \tilde{C}_3 \sum_{k=0}^K \|\bm{\eta}_k\|^4+ \tilde{C}_0.
	\end{aligned}
\end{equation}
Before proceeding to the subsequent analysis, we establish the following inequality
\begin{align}\label{eq:grad_bound}
	&\|\operatorname{grad} f(\bm{x}_{k+1}) - \operatorname{grad} f(\bm{x}_{k})\| \notag\\
	&\leq 4L \|\bm{x}_{k} - \bar{\bm{x}}_{k}\| + 4L \|\bm{s}_{k}\|+2L\|\bm{\eta}_k\|.
\end{align}
By invoking the Lipschitz continuity of $\operatorname{grad} f$ yields $\|\operatorname{grad} f(\bm{x}_{k+1})  - \operatorname{grad} f(\bm{x}_{k})\| \leq L \|\bm{x}_{k+1} - \bm{x}_{k}\|$. To bound the term $\|\bm{x}_{k+1} - \bm{x}_{k}\|$, we apply the triangle inequality and the update rules as follows
\begin{align*}
	&\|\bm{x}_{k+1} -\bm{x}_{k}\| \\
	&\leq \|\bm{x}_{k+1} -\bm{y}_{k}\| + \|\bm{x}_{k} -\bm{y}_{k}\|\\
	&\leq \|\mathcal{P}_{\mathcal{M}^n}(\bm{y}_{k}+\bm{\eta}_{k}) -\bm{y}_{k}-\bm{\eta}_{k }\| +\|\bm{\eta}_{k }\|\\
	&\quad  +\|\mathcal{P}_{\mathcal{M}^n}(\mathbf{W}\bm{x}_{k}+\bm{s}_{k})-\bm{x}_{k}+(I-\mathbf{W})\bm{x}_{k}-\bm{s}_{k}\|\\
	&\quad  + \|(I-\mathbf{W})\bm{x}_{k}-\bm{s}_{k}\|\\
	&\leq \| \bm{x}_{k}-\bm{y}_{k}-\bm{\eta}_{k }\|+ \|\mathcal{P}_{\mathcal{M}^n}(\mathbf{W}\bm{x}_{k}+\bm{s}_{k})-\mathbf{W} \bm{x}_{k}-\bm{s}_{k}\|\\
	&\quad  + \|(I-\mathbf{W})\bm{x}_{k}-\bm{s}_{k}\|+\|\bm{\eta}_{k }\|\\
	&\leq 2(\|\bm{x}_{k}-\mathbf{W} \bm{x}_{k}-\bm{s}_{k}\|+ \|(I-\mathbf{W})\bm{x}_{k}-\bm{s}_{k}\|)+ 2\|\bm{\eta}_{k }\|\\
	&\leq 2(\|\bm{x}_{k}-\mathbf{W} \bm{x}_{k}-\bm{s}_{k}\|+ \|(I-\mathbf{W})\bm{x}_{k}-\bm{s}_{k}\|)+ 2\|\bm{\eta}_{k }\|\\
	&\leq 4\|\bm{x}_{k}-\bar{\bm{x}}_{k}\|+ 4\|\bm{s}_{k }\|+2\|\bm{\eta}_{k }\|,
\end{align*}
where the second inequality follows from the definition of $\bm{x}_{k+1}$ and $\bm{y}_{k+1}$, since $\mathcal{P}_{\mathcal{M}^n}(\mathbf{W}\bm{x}_{k}+\bm{s}_{k})$ is the point on the manifold closest to $\mathbf{W}\bm{x}_{k}+\bm{s}_{k}$, and analogously for $\bm{y}_{k}+\bm{\eta}_{k}$. With the inequality \eqref{eq:grad_bound} established, we are now ready to proceed with the primary analysis.

Then, we obtain
\begin{equation}\label{gradient bound}
	\begin{aligned}
		&\alpha^2 \sum_{k=0}^K \|\text{grad} f(\bm{x}_{k+1}) - \text{grad} f(\bm{x}_k)\|^2  \\
		&\quad \le \alpha^2 \left( \sum_{k=0}^K 32L^2\|\bm{x}_k - \bar{\bm{x}}_k\|^2 + 32L^2\|\mathbf{s}_k\|^2 + 8L^2\|\bm{\eta}_k\|^2 \right).
	\end{aligned}
\end{equation}
Let $E_k := \|\bm{x}_k - \bar{\bm{x}}_k\|^2 + \|\bm{s}_k\|^2$. Substituting \eqref{gradient bound} into \eqref{xie bound} and using the inequality $\|a+b\|^2 \ge \frac{1}{2}\|a\|^2 - \|b\|^2$ to decouple $\bm{s}_{k+1}$ and $\hat{\bm{g}}_{k+1}$ on the left-hand side, we obtain
\begin{align*}
	&\frac{1}{2}\sum_{k=0}^K E_{k+1} - \alpha^2 \sum_{k=0}^K \|\hat{\bm{g}}_{k+1}\|^2 \\
	&\le \alpha^2 \tilde{C}_1 \sum_{k=0}^K \|\hat{\bm{g}}_k\|^2 + 32\alpha^2 L^2 \tilde{C}_1 \sum_{k=0}^K E_k \\
	&\quad + (8L^2 \alpha^2 \tilde{C}_1 + \tilde{C}_2) \sum_{k=0}^K \|\bm{\eta}_k\|^2 + \tilde{C}_3 \sum_{k=0}^K \|\bm{\eta}_k\|^4 + \tilde{C}_0.
\end{align*}
Rearranging terms and aligning the summation indices
\begin{align*}
	&(1 - 64\alpha^2 L^2 \tilde{C}_1) \sum_{k=1}^K E_k \\
	&\le 2\alpha^2 (\tilde{C}_1 + 1) \sum_{k=1}^K \|\hat{\bm{g}}_k\|^2 + (16L^2 \alpha^2 \tilde{C}_1 + 2\tilde{C}_2) \sum_{k=1}^K \|\bm{\eta}_k\|^2 \\
	&\quad + 2\tilde{C}_3 \sum_{k=0}^K \|\bm{\eta}_k\|^4 + 2\tilde{C}_0.
\end{align*}
Provided that the stepsize satisfies $\alpha \le \frac{1}{8L\sqrt{\tilde{C}_1}}$, the coefficient $(1 - 64\alpha^2 L^2 \tilde{C}_1)$ is positive. Dividing by this coefficient yields the final cumulative bound:
\begin{equation}\label{eq:cumulative_final}
	\sum_{k=1}^K E_k \le \alpha^2 C_1 \sum_{k=1}^K \|\hat{\bm{g}}_k\|^2 + C_2 \sum_{k=1}^K \|\bm{\eta}_k\|^2  + C_3 \sum_{k=0}^K \|\bm{\eta}_k\|^4 + C_0,
\end{equation}
where the constants $C_0, C_1, C_2, C_3$ are defined as in \eqref{constant2},
\begin{equation}\label{constant2}
	\begin{aligned}[b]
		C_0 &= \frac{2\tilde{C}_0}{1 - 64\alpha^2 L^2 \tilde{C}_1}, \quad
		C_1 = \frac{2(\tilde{C}_1 + 1)}{1 - 64\alpha^2 L^2 \tilde{C}_1}, \quad\\
		C_2 &= \frac{16L^2 \alpha^2 \tilde{C}_1 + \tilde{C}_2}{1 - 64\alpha^2 L^2 \tilde{C}_1}, \quad
		C_3 = \frac{\tilde{C}_3}{1 - 64\alpha^2 L^2 \tilde{C}_1}.
	\end{aligned}
\end{equation}
To establish the uniform bound, we revisit the recursive inequality \eqref{X_k bound}
\begin{align*}
	&\|\mathbf{X}_{k+1}\| \\
	&\le \bar{\nu}^{k+1}\|\mathbf{X}_0\| + \sum_{l=0}^{k} \bar{\nu}^{k-l} \left( 2\alpha \|\mathbf{C}_l\| + \frac{\|\mathbf{d}_l\|}{1-3\delta} + Q\|\mathbf{d}_l\|^2 \right).
\end{align*}
Using the boundedness of gradients, specifically $\|\mathbf{C}_l\| \le \sqrt{5n}L_g$ and $\|\mathbf{d}_l\| \le \tau\sqrt{n}L_r$, and the sum of the geometric series $\sum_{j=0}^{\infty} \bar{\nu}^j = \frac{1}{1-\bar{\nu}}$, there exists a constant $C > 0$ such that
\[
\frac{1}{\sqrt{n}} \|\mathbf{X}_{k+1}\| \le C (L\alpha + L_r\tau + L_r^2\tau^2).
\]
Squaring this result implies $\frac{1}{n} \|\mathbf{X}_{k+1}\|^2 \le \mathcal{O}(\alpha^2 + \tau^2)$. Finally, we get
\begin{align*}
	&\frac{1}{n}(\|\bm{x}_{k} - \bar{\bm{x}}_{k}\|^2+\|\bm{s}_{k}\|^2)\\
	&\leq \frac{1}{n}(\|\mathbf{X}_k\|^2 + 2\alpha\|\bm{s}_k\|\|\hat{\bm{g}}_k\| + \alpha^2\|\hat{\bm{g}}_k\|^2) \\
	&\leq (2C^2+2)(L^2\alpha^2+L_r^2\tau^2+L_r^4\tau^4).
\end{align*}
This confirms the uniform boundedness of the consensus and tracking errors as stated in \eqref{x,s bound} and \eqref{x,s bound2}. \quad \qed

\section{Proof of Lemma \ref{l8}}\label{Appendix E}
By the convexity of $r$ and Jensen's inequality, we have
\[
r(\hat{y}_{k} + \hat{\eta}_k) \leq \frac{1}{n} \sum_{i=1}^{n} r(y_{i,k} + \eta_{i,k}).
\]
Furthermore, since $r$ is $L_r$-Lipschitz continuous, we have the following bound
\begin{align*}
	&r(\hat{x}_{k+1}) - r(\hat{y}_{k}  + \hat{\eta}_k) \\
	&\leq L_r \left\|\hat{x}_{k+1} - \hat{y}_{k}  - \hat{\eta}_k \right\| \\
	&\leq L_r \|\hat{x}_{k+1} - \hat{y}_{k} \| + \frac{L_r}{\sqrt{n}}\|\bm{\eta}_k\|.
\end{align*}
We utilize the strong convexity of the local approximation $g_i^{(k)}$. Based on \eqref{g convex}, the optimality condition implies
\begin{align*}
	&g_{i,{k}} \left( \mathcal{P}_{T_{y_{i,k}}\mathcal{M}} (\bar{x}_{k} - x_{i,k}) \right) - g_{i,{k}}(\eta_{i,k}) \\
	&\geq \frac{1}{2\tau} \left\| \mathcal{P}_{T_{y_{i,k}}\mathcal{M}} (\bar{x}_{k} - x_{i,k}) - \eta_{i,k} \right\|^2 \geq 0,
\end{align*}
Combining this with the definition of the proximal update, we obtain an upper bound for the regularization term $r$
\begin{align}
	&	r(y_{i,k} +\eta_{i,k})\notag\\ 
	&\leq r\left( y_{i,k}+ \mathcal{P}_{y_{i,k}} (\bar{x}_{k} - x_{i,k}) \right) 	- \frac{1}{2\tau} \left\| \eta_{i,k} \right\|^2 \nonumber\\
	&\quad  	+ \frac{1}{2\tau} \left\| \mathcal{P}_{y_{i,k}} (\bar{x}_{k} - x_{i,k}) \right\|^2 . \label{r,eq}
\end{align}
Summing the inequality \eqref{r,eq} over $i = 1, \dots, n$ and dividing by $n$ gives
\begin{equation}\label{r eq}
	\begin{aligned}[b]
		&\frac{1}{n} \sum_{i=1}^{n}	r(y_{i,k} +\eta_{i,k}) \\
		&\leq 	\frac{1}{n} \sum_{i=1}^{n} r\left( y_{i,k}+ \mathcal{P}_{y_{i,k}} (\bar{x}_{k} - x_{i,k}) \right) 	-	  \frac{1}{2n\tau} \left\| \bm{\eta}_k\right\|^2 \\
		&\quad  	+ 	\frac{1}{2n\tau} \left\| \bar{\bm{x}}_{k} - \bm{x}_{k} \right\|^2 .
	\end{aligned}
\end{equation}

We first analyze the term \( r\bigl( y_{i,k}+ \operatorname{Proj}_{T_{y_{i,k}}\mathcal{M}} (\bar{x}_{k} - x_{i,k}) \bigr) \).
By adding and subtracting \( r(\bar{x}_{k}) \) and \( r(\hat{x}_{k}) \), and then using the Lipschitz continuity of \( r \) and properties of the projection operator, we obtain
\begin{align*}
	&	r\left( y_{i,k}+ \mathcal{P}_{T_{y_{i,k}}\mathcal{M}} (\bar{x}_{k} - x_{i,k}) \right)\\
	&= 	r\left( y_{i,k}+ \mathcal{P}_{T_{y_{i,k}}\mathcal{M}} (\bar{x}_{k} - x_{i,k}) \right)+r(\bar{x}_{k})-r(\bar{x}_{k})\\
	&\quad +r(\hat{x}_{k})-r(\hat{x}_{k})\\
	&\leq L_r\|-y_{i,k}- \mathcal{P}_{T_{y_{i,k}}\mathcal{M}} (\bar{x}_{k} - x_{i,k})+\hat{x}_{k}\|+L_r\|\hat{x}_{k}-\bar{x}_{k}\|\\
	&\quad +r(\bar{x}_{k})\\
	&\leq L_r\|\hat{x}_{k}-y_{i,k}\|+L_r\|\bar{x}_{k} - x_{i,k}\| +L_r\|\hat{x}_{k}-\bar{x}_{k}\| +r(\bar{x}_{k})\\
	&\leq  L_r\|\hat{x}_{k}-x_{i,k}\|+ L_r\|{x}_{i,k}-y_{i,k}\|
	+L_r\|\bar{x}_{k} - x_{i,k}\| \\&\quad +L_r\|\hat{x}_{k}-\bar{x}_{k}\| 
	+r(\bar{x}_{k}).
\end{align*}
Hence, \eqref{r eq} is further relaxed to
\begin{equation}\label{r eq2}
	\begin{aligned}[b]
		&\frac{1}{n} \sum_{i=1}^{n}	r(y_{i,k} +\eta_{i,k}) \\
		&\leq 		-	  \frac{1}{2n\tau} \left\| \bm{\eta}_k\right\|^2 	+ 	\frac{1}{2n\tau} \left\| \bar{\bm{x}}_{k} - \bm{x}_{k} \right\|^2 +\frac{L_r}{\sqrt{n}}\|\bar{\bm{x}}_{k} - \bm{x}_{k}\|\\
		&\quad +\frac{L_rM}{n} \left\| \bar{\bm{x}}_{k} - \bm{x}_{k} \right\|^2 +\frac{L_r}{\sqrt{n}}\left\| \bar{\bm{x}}_{k} - \bm{x}_{k} \right\|\\
		&\quad +\frac{1}{n} \sum_{i=1}^{n} L_r\|{x}_{i,k}-y_{i,k}\|.
	\end{aligned}
\end{equation}
Proceeding similarly to Lemma~\ref{A2}, we define $d_i = \nabla \phi_i(\bm{x}_k) - s_{i,k}$ and bound the term $\|{x}_{i,k} - y_{i,k}\|$ as follows,
\begin{equation}\label{r eq3}
	\begin{aligned}
		&\frac{1}{n} \sum_{i=1}^{n}\|{x}_{i,k}-y_{i,k}\|\\
		&\leq 	\frac{1}{n} \sum_{i=1}^{n}[\|\mathcal{P}_\mathcal{M}(x_{i,k} - d_{i,1} - d_{i,2}) - x_{i,k} + d_{i,1}\|+\|d_{i,1}\|]\\
		&\leq \frac{1}{n} \sum_{i=1}^{n}  (Q\|d_i\|^2+\|d_i\|)\\
		&\leq  \frac{8Q}{n} \|\bm{x}_k-\bar{\bm{x}}_k\|^2+ \frac{1}{\sqrt{n}}\|\bm{x}_k-\bar{\bm{x}}_k\|+\frac{2}{n}\|\bm{s}_k\|^2.
	\end{aligned}
\end{equation}
Combining \eqref{r eq2} and \eqref{r eq3}, we arrive at \eqref{r desent}. \quad \qed
\section{Proof of Theorem \ref{t1}}\label{Appendix F}
We begin by establishing an upper bound on the $\|\bar{x}_{k+1} - \bar{x}_k\|$. Using the triangle inequality, we have
$
\| \bar{x}_{k+1} - \bar{x}_k \| \le \|\hat{x}_{k+1}-\hat{x}_k\| + \|\bar{x}_{k+1}-\hat{x}_{k+1}\| + \|\bar{x}_k-\hat{x}_k\|.
$
Applying Lemma~\ref{A2} to the projection deviation terms, this relaxes to
\begin{align}\label{eq:bar_x_diff_inter}
	&\| \bar{x}_{k+1} - \bar{x}_k \| \notag\\
	&\le \|\hat{x}_{k+1}-\hat{x}_k\| + \frac{M}{n} \left( \|\bm{x}_{k+1}-\bar{\bm{x}}_{k+1}\|^2 + \|\bm{x}_{k}-\bar{\bm{x}}_{k}\|^2 \right).
\end{align}
It remains to bound the term $\|\hat{x}_{k+1}-\hat{x}_k\|$. Using the update rules and rearranging terms, we obtain
\begin{align}\label{t1-l2 eq}
	&\|\hat{x}_{k+1} - \hat{x}_{k}\| \notag \\
	&= \left\| \frac{1}{n} \sum_{i=1}^n (y_{i,k} - x_{i,k}) - \frac{1}{n} \sum_{i=1}^n (y_{i,k} - x_{i,k+1}) \right.\notag  \\
	&\quad \left.+ \frac{1}{n} \sum_{i=1}^n \left( \operatorname{grad} \phi_i(\bm{x}_k) - \mathcal{P}_{T_{x_{i,k}}\mathcal{M}}(s_{i,k}) \right) \right\|\notag  \\
	&\leq \frac{Q}{n} \sum_{i=1}^n \| s_{i,k} - \nabla \phi_i(\bm{x}_k) \|^2 + \frac{1}{n} \sum_{i=1}^n \|s_{i,k}\|\notag  \\
	&\quad + \frac{1}{n} \left\| \sum_{i=1}^n \operatorname{grad} \phi_i(\bm{x}_k) \right\| + \left\| \frac{1}{n} \sum_{i=1}^n y_{i,k} - \frac{1}{n} \sum_{i=1}^n x_{i,k+1} \right\|.
\end{align}
We now bound the terms on the right-hand side of \eqref{t1-l2 eq}. Using the property $\|\mathcal{P}_{T_{x_{i,k}}\mathcal{M}}(s_{i,k})\|\le \|s_{i,k}\|$, we have
$
\frac{1}{n}\sum_{i=1}^n \|\mathcal{P}_{T_{x_{i,k}}\mathcal{M}}(s_{i,k})\| \le \frac{1}{n}\sum_{i=1}^n \|s_{i,k}\| \le \frac{1}{\sqrt{n}}\|\bm{s}_k\|.
$
Furthermore, using $\|\nabla \phi(\bm{x}_k)\|\le 2\|\bm{x}_k-\bar{\bm{x}}_k\|$ together with \eqref{phi bound}, we obtain
\begin{align}\label{A3 eq}
	&\|\hat{x}_{k+1} - \hat{x}_{k}\| \notag\\
	&\leq \frac{8Q + 2\sqrt{n}L_p}{n} \|\bm{x}_{k} - \bar{\bm{x}}_{k}\|^2 + \frac{2Q}{n} \|\bm{s}_{k}\|^2 \notag\\
	&\quad + \frac{1}{\sqrt{n}} \|\bm{s}_{k}\| + \left\| \frac{1}{n} \sum_{i=1}^n y_{i,k} - \frac{1}{n} \sum_{i=1}^n x_{i,k+1} \right\|.
\end{align}
By Jensen's inequality and the definition of the projection operator,
\begin{equation}\label{A3 eq2}
	\begin{aligned}[b]
		&\left\| \frac{1}{n} \sum_{i=1}^n y_{i,k} - \frac{1}{n} \sum_{i=1}^n x_{i,k+1} \right\| \\
		&\leq \frac{1}{n} \sum_{i=1}^n \| y_{i,k} - x_{i,k+1} \| \\
		&\leq \frac{1}{n} \sum_{i=1}^n  \|  \mathcal{P}_\mathcal{M}(y_{i,k} + \eta_{i,k}) - y_{i,k} - \mathcal{P}_{T_{y_{i,k}\mathcal{M}}}(\eta_{i,k}) \| \\
		&\quad +  \frac{1}{n} \sum_{i=1}^n   \| \mathcal{P}_{T_{y_{i,k}\mathcal{M}}}(\eta_{i,k}) \| ) \\
		&\leq \frac{1}{n} \sum_{i=1}^n \left( Q \|\eta_{i,k}\|^2 + \|\eta_{i,k}\| \right) \\
		&= \frac{Q}{n} \|\bm{\eta}_{k}\|^2 + \frac{1}{\sqrt{n}} \|\bm{\eta}_{k}\|.
	\end{aligned}
\end{equation}
Substituting \eqref{A3 eq2} into \eqref{A3 eq}, and then combining the result with \eqref{eq:bar_x_diff_inter}, we arrive at the bound
\begin{align}\label{t1-l1 eq}
	&\| \bar{x}_{k+1} - \bar{x}_k \|\notag\\
	&\leq \frac{M}{n}\!\left( \|\bm{x}_{k+1} - \bar{\bm{x}}_{k+1}\|^2 + \|\bm{x}_{k} - \bar{\bm{x}}_{k}\|^2 \right)\notag \\
	&\quad + \frac{8Q + 2\sqrt{n}L_p}{n}\|\bm{x}_{k} - \bar{\bm{x}}_{k}\|^2
	+ \frac{2Q}{n}\|\bm{s}_{k}\|^2 + \frac{1}{\sqrt{n}}\|\bm{s}_{k}\|\notag \\
	&\quad + \frac{Q}{n}\|\bm{\eta}_{k}\|^2 + \frac{1}{\sqrt{n}}\|\bm{\eta}_{k}\|.
\end{align}

Invoking Lemma~\ref{l5}, we have
\begin{align*}
	&f(\bar{x}_{k+1}) \\
	&\le f(\bar{x}_{k}) + \frac{8LQ + 9L}{n}\left\| \bm{x}_k - \bar{\bm{x}}_k \right\|^2 + \frac{2LQ+2L}{n} \left\| \bm{s}_k\right\|^2\\
	& \quad -\alpha\|\hat{g}_k\|^2 + (4Q\tau^2 L_r^2+2\tau L_r)\left\| \hat{g}_k \right\| + \frac{\alpha^2 L}{2} \| \hat{g}_k \|^2 \\
	&\quad + \frac{M^2C}{n} \left( \| \bm{x}_{k+1} - \bar{\bm{x}}_{k+1} \|^2 + \| \bm{x}_{k} - \bar{\bm{x}}_{k} \|^2 \right) \\
	& \quad + \frac{L}{n} \| \bm{x}_k - \bar{\bm{x}}_k \|^2 + \frac{3L}{4} \left[ \Psi_k \right]^2,
\end{align*}
where $\Psi_k$ represents the collection of higher-order terms from \eqref{t1-l1 eq}. Using the bound $\frac{1}{n}(\|\bm{x}_{k} - \bar{\bm{x}}_{k}\|^2+\|\bm{s}_{k}\|^2) \leq (2C+2)(L^2\alpha^2+L_r^2\tau^2+L_r^4\tau^4)$, and combining this with Lemma \ref{l8}, we recall that $h(\bar{x}_k)=f(\bar{x}_k)+r(\bar{x}_k)$. This yields the descent inequality for the composite objective function
By the triangle inequality $
\| \bar{x}_{k+1} - \bar{x}_k \|
\le \|\hat{x}_{k+1}-\hat{x}_k\| + \|\bar{x}_{k+1}-\hat{x}_{k+1}\| + \|\bar{x}_k-\hat{x}_k\|,
$ applying Lemma~\ref{A2} to the projection deviations yields
\begin{equation}\label{h desent}
	\begin{aligned}
		&h(\bar{x}_{k+1})\\
		&= f(\bar{x}_{k+1}) + r(\bar{x}_{k+1}) - r(\hat{x}_{k+1}) + r(\hat{x}_{k+1})\\
		&\leq f(\bar{x}_{k+1}) + r(\hat{x}_{k+1}) + L\|\bar{x}_{k+1}-\hat{x}_{k+1}\| \\
		&\leq h(\bar{x}_k) - \left(\alpha-\frac{\alpha^2L}{2}\right)\|\hat{g}_k\|^2 + D_1\|\bm{s}_{k}\|^2\\
		&\quad + D_2\|\bm{x}_{k} - \bar{\bm{x}}_{k}\|^2 + D_3\|\bm{x}_{k+1} - \bar{\bm{x}}_{k+1}\|^2\\
		&\quad - \left(\frac{1}{2n\tau}-\frac{4}{n}\right)\|\bm{\eta}_k\|^2 + D_4,
	\end{aligned}
\end{equation}
where the auxiliary constants $D_i,i=1,2,3$ are defined as follows
\begin{equation}\label{constant}
	\begin{aligned}[b]
		D_1 &= \frac{\mathcal{A}\mathcal{B} + 4 + 2LQ + 2L}{n}, \\
		D_2 &= \frac{8LQ + 9L + M^2C}{n} + \frac{\mathcal{C}^2 \mathcal{A} \mathcal{B}}{n}, \\
		D_3 &= \frac{M^2C}{n} + \frac{\mathcal{C}^2 \mathcal{A} \mathcal{B}}{n}, \quad 
		D_4 = \frac{64QL_r^4\tau^4}{n^2}\\
		\mathcal{A}&= 12QCL + 12QL, \quad 
		\mathcal{B} = L^2\alpha^2 + L_r^2\tau^2 + L_r^4\tau^4, \\
		\mathcal{C} &= M + 8Q + 2\sqrt{n}L_p.
	\end{aligned}
\end{equation}
Summing \eqref{h desent} over $k = 0, \dots, K$ and choosing the stepsizes such that
\begin{align*}
	\alpha 
	&\le 
	\min\!\left\{
	1,\;
	\frac{1}{4C_1(D_1+D_2+D_3)}
	\right\}, \\
	\tau 
	&\le 
	\min\!\left\{
	\frac{1}{16+4nC_2},\;
	\left(\frac{1}{32nL_r^2C_3}\right)^{1/3}
	\right\}.
\end{align*}
We obtain
\begin{align*}	
	&h(\bar{x}_{K+1}) \\
	&\leq h(\bar{x}_{0}) - \left(\frac{\alpha}{2}-C_1(D_1+D_2+D_3)\alpha^2\right) \sum_{k=0}^{K}\|\hat{g}_k\|^2 \\
	&\quad - \left(\frac{1}{2n\tau}-\frac{4}{n}-C_2-4\tau^2L_r^2C_3\right) \sum_{k=0}^{K}\|\bm{\eta}_k\|^2 + D_4,
\end{align*}	
which implies that
\begin{align*}	
	\alpha \sum_{k=0}^{K}\|\hat{g}_k\|^2 + \frac{1}{8n\tau} \sum_{k=0}^{K}\|\bm{\eta}_k\|^2 \leq 8(h(\bar{x}_0)-h^*+D_4),
\end{align*}	
where $h^* = \underset{x\in\mathcal{M}}{\text{min}} h(x)$. Defining $C_{\text{min}} = \text{min}\{\alpha, \frac{1}{8n\tau}\}$, we arrive at
\begin{align*}	
	\underset{{0\le k\le K}}{\text{min}} \{\|\hat{g}_k\|^2+\|\bm{\eta}_k\|^2\} \leq \frac{8(h(\bar{x}_0)-h^*+D_4)}{C_{\text{min}}K}.
\end{align*}			
Using the bound \eqref{x,s bound}, the consensus error satisfies
\begin{equation}\label{consensus}
	\underset{{0\le k\le K}}{\text{min}} \|\bm{x}_k - \bar{\bm{x}}_k\|^2 \leq \mathcal{O}\left(\frac{1}{K}\right).
\end{equation}

Finally, noting that $\|\operatorname{grad} f(\bar{x}_k)\|^2 \le 2\|\hat{g}_k\|^2 + 2\|\operatorname{grad} f(\bar{x}_k) -\hat{g}_k\|^2 \le 2\|\hat{g}_k\|^2 + \frac{2L^2}{n}\|\bm{x}_k - \bar{\bm{x}}_k\|^2$, we conclude that
\begin{equation}\label{rate}
	\underset{{0\le k\le K}}{\text{min}} \text{max} \{\|\operatorname{grad} f(\bar{x}_k)\|^2,\|\bm{\eta}_k\|^2\} \le \mathcal{O}\left(\frac{1}{K}\right), 
\end{equation}	
which completes the proof.\quad \qed
%

\section{Proof of Theorem \ref{t2}}\label{Appendix G}

Owing to the compactness of $\mathcal{M}$, the sequence $\{\bm{x}_k\}$ is bounded. 
By the Bolzano–Weierstrass theorem \cite{BW}, which states that every bounded sequence possesses a convergent subsequence, there exists an accumulation point
$\bm{x}^\ast$ of $\{\bm{x}_k\}$. The consensus relation~\eqref{consensus} implies that
\begin{align*}
	\bm{x}^\ast = (1_n\otimes I_d)\,\bar x^\ast,
\end{align*}
for some $\bar x^\ast \in \mathcal{M}$.

The convergence rate established in~\eqref{rate} guarantees the existence of a subsequence $\{k_\ell\}$ such that the following limit holds
\begin{align*}
	\lim_{\ell \to \infty} \max\Bigl\{
	\|\operatorname{grad} f(\bar{x}_{k_\ell})\|^2,\,
	\|\bm{x}_{k_\ell}-\bar{\bm{x}}_{k_\ell}\|^2,\,
	\|\bm{\eta}_{k_\ell}\|^2
	\Bigr\} = 0.
\end{align*}
Consequently, for every node $i$, both the consensus error $x_{i,k_\ell} - \bar x_{k_\ell}$ and the variable $\eta_{i,k_\ell}$ vanish as $\ell$ tends to infinity. Since the manifold $\mathcal{M}$ is compact, the sequence $\{x_{i,k}\}$ must possess convergent subsequences. Without loss of generality, we assume that $x_{i,k_\ell}$ converges to an accumulation point $\bar x^\ast \in \mathcal{M}$. This, combined with the fact that $x_{i,k_\ell} - \bar x_{k_\ell}$ approaches zero, directly implies that $\bar x_{k_\ell}$ converges to $\bar x^\ast$ and the joint vector $\bm{x}_{k_\ell}$ approaches $(1_n \otimes I_d)\bar x^\ast$.

According to Algorithm \ref{algorithm_extra}, the updates for the variables are given by
\begin{align*}
	y_{i,k} &= \mathcal{P}_{\mathcal{M}}\left(\sum_{j=1}^n w_{ij}x_{j,k}+s_{i,k}\right), \\
	x_{i,k+1} &= \mathcal{P}_{\mathcal{M}}(y_{i,k}+\eta_{i,k}).
\end{align*}
Given that $\mathcal{M}$ is compact and the projection operator $\mathcal{P}_{\mathcal{M}}$ is smooth, it is locally Lipschitz continuous on a compact neighborhood containing all iterates. Therefore, there exists a constant $L = \frac{R}{R - \tau}$ such that the following inequality holds for all $z_1, z_2$, 
\begin{align*}
	\operatorname{dist}_{\mathcal{M}}
	\bigl(\mathcal{P}_{\mathcal{M}}(z_1),\mathcal{P}_{\mathcal{M}}(z_2)\bigr)
	\le L \|z_1-z_2\|.
\end{align*}

Applying this Lipschitz property with $z_1=y_{i,k_\ell}+\eta_{i,k_\ell}$ and $z_2=y_{i,k_\ell}$, we obtain:
\begin{align}\label{dist}
	\operatorname{dist}_{\mathcal{M}}(x_{i,k_\ell+1},y_{i,k_\ell})
	\le \frac{R}{R - \tau} \|\eta_{i,k_\ell}\|.
\end{align}
As $\ell$ goes to infinity, the right-hand side of \eqref{dist} vanishes. Since $x_{i,k_\ell+1}$ converges to $\bar x^\ast$, it follows that $y_{i,k_\ell}$ also tends to $\bar x^\ast$, and the collective variable $\bm{y}_{k_\ell}$ converges to $	\bm{x}^\ast$. Consequently, the sequences $\{\bm{y}_k\}$ and $\{\bm{x}_k\}$ share the same set of accumulation points.

For an arbitrary node $i$, the optimality condition of the subproblem for $\eta_{i,k}$ is expressed as:
\begin{align*}
	0 \in 
	\mathcal{P}_{T_{y_{i,k}}\mathcal{M}}
	\Bigl(
	\frac{1}{\tau}\eta_{i,k} + \partial h(y_{i,k}+\eta_{i,k})
	\Bigr).
\end{align*}
Equivalently, this can be rewritten as:
\begin{align*}
	&\mathcal{P}_{T_{y_{i,k}}\mathcal{M}}
	\bigl(\nabla f(y_{i,k}+\eta_{i,k})\bigr)
	-\frac{1}{\tau}\eta_{i,k}\\
	&\in
	\mathcal{P}_{T_{y_{i,k}}\mathcal{M}}
	\bigl(\partial h(y_{i,k}+\eta_{i,k})\bigr),
\end{align*}
which implies that there exists a normal vector $E_{i,k}\in\mathcal{N}_{y_{i,k}}\mathcal{M}$ such that:
\begin{align*}
	\mathcal{P}_{T_{y_{i,k}}\mathcal{M}}
	\bigl(\nabla f(y_{i,k}+\eta_{i,k})\bigr)
	-\frac{1}{\tau}\eta_{i,k}
	+E_{i,k}
	\in \partial h(y_{i,k}+\eta_{i,k}).
\end{align*}

Recalling the previously established subsequential convergence, we have $(y_{i,k_\ell}, \eta_{i,k_\ell})$ converges to $(\bar x^\ast, 0)$ as $\ell$ tends to $\infty$. This convergence implies that, for any $\beta > 0$, the sequence $\{y_{i,k_\ell} + \eta_{i,k_\ell}\}$ eventually enters and remains within the $\beta$-neighborhood $\mathcal{R}(\beta)$ of $\bar x^\ast$. Proposition 2.1.2 in~\cite{clarke} guarantees that the generalized gradient $\partial h$ is locally bounded, ensuring the boundedness of the set $\{H \mid H \in \partial h(y),\, y \in \mathcal{R}(\beta)\}$. Consequently, the sequence of normal vectors $\{E_{i,k_\ell}\}$ is bounded. By the Bolzano-Weierstrass theorem, $\{E_{i,k_\ell}\}$ possesses a convergent subsequence with limit $E_i^\ast$; for notational simplicity, we relabel this subsequence using the same index $\ell$.

Applying Remark 1(ii) in \cite{Bolte}, which establishes the closedness of the graph of the generalized subdifferential under function value convergence, we deduce that $E_i^\ast \in \partial h(\bar x^\ast)$ since the continuity of $h$ ensures $h(y_{i,k_\ell} + \eta_{i,k_\ell})$ converges to $h(\bar x^\ast)$. To relate this limit to the manifold structure, we invoke the smoothness of $\mathcal{M}$, which guarantees the continuity of the projection operator onto the normal bundle. Since $E_{i,k_\ell} \in \mathcal{N}_{y_{i,k_\ell}}\mathcal{M}$ and $y_{i,k_\ell}$ converges to $\bar x^\ast$, the limit $E_i^\ast$ satisfies:
\begin{equation*}
	E_i^\ast = \lim_{\ell \to \infty} \mathcal{P}_{\mathcal{N}_{y_{i,k_\ell}}\mathcal{M}}(E_{i,k_\ell}) = \mathcal{P}_{\mathcal{N}_{\bar x^\ast}\mathcal{M}}(E_i^\ast),
\end{equation*}
implying that its tangential component vanishes, i.e., $\mathcal{P}_{T_{\bar x^\ast}\mathcal{M}}(E_i^\ast) = \bm{0}$. Consequently, substituting this into the first-order inclusion yields:
\begin{equation*}
	0 \in \mathcal{P}_{T_{\bar x^\ast}\mathcal{M}} \bigl(\nabla f(\bar x^\ast) + \partial r(\bar x^\ast)\bigr),
\end{equation*}
thereby demonstrating that $\bar x^\ast$ is a stationary point of problem~\eqref{composite}. \quad \qed

	\bibliography{autosam}
	

	\begin{biography}{Yongyang Xiong}{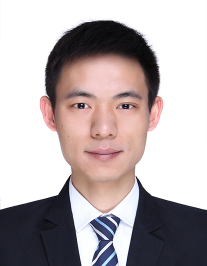}
received the B.S. degree in information and computational science, the M.E. and Ph.D. degrees in control science and engineering from Harbin Institute of Technology, Harbin, China, in 2012, 2014, and 2020, respectively. From 2017 to 2018, he was a joint Ph.D. student with the School of Electrical and Electronic Engineering, Nanyang Technological University, Singapore. From 2021 to 2024, he was a Postdoctoral Fellow with the Department of Automation, Tsinghua University, Beijing, China. Currently, he is an associate professor with the School of Intelligent Systems Engineering, Sun Yat-sen University, Shenzhen, China. His current research interests include networked control systems, distributed optimization and learning, multi-agent reinforcement learning and their applications. 		
	\end{biography}
	
	\begin{biography}{Chen Ouyang}{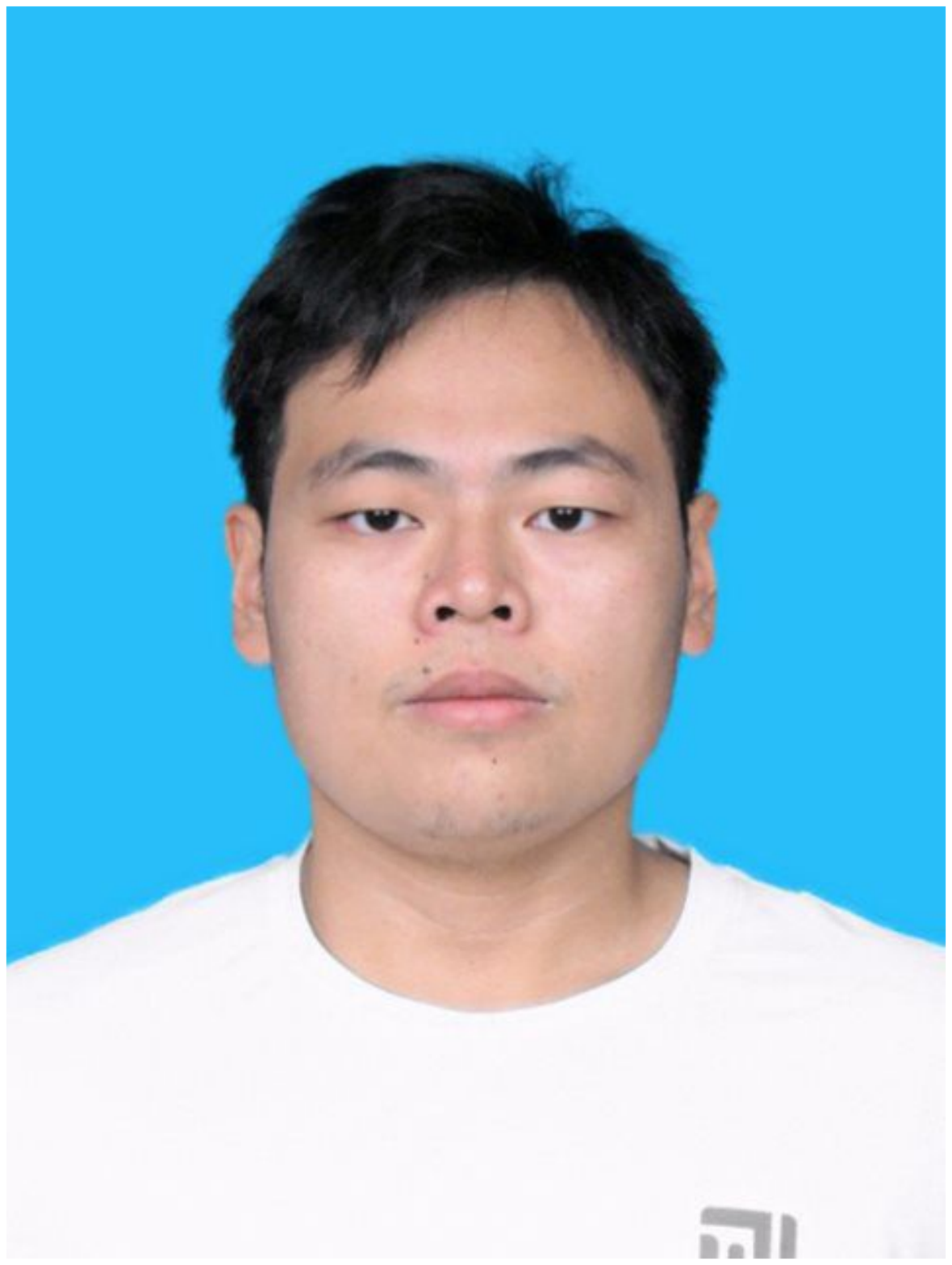}
		received the B.S. degree in information and computational science, the M.E. degrees in School of Mathematics and Information Science from Guangxi University, Nanning, China, in 2022, 2025, respectively. He is now pursuing his Ph.D. in the School of Intelligent Systems Engineering, Sun Yat-sen University. His current research interests include distributed optimization and machine learning.
	\end{biography}

		\begin{biography}{Keyou You}{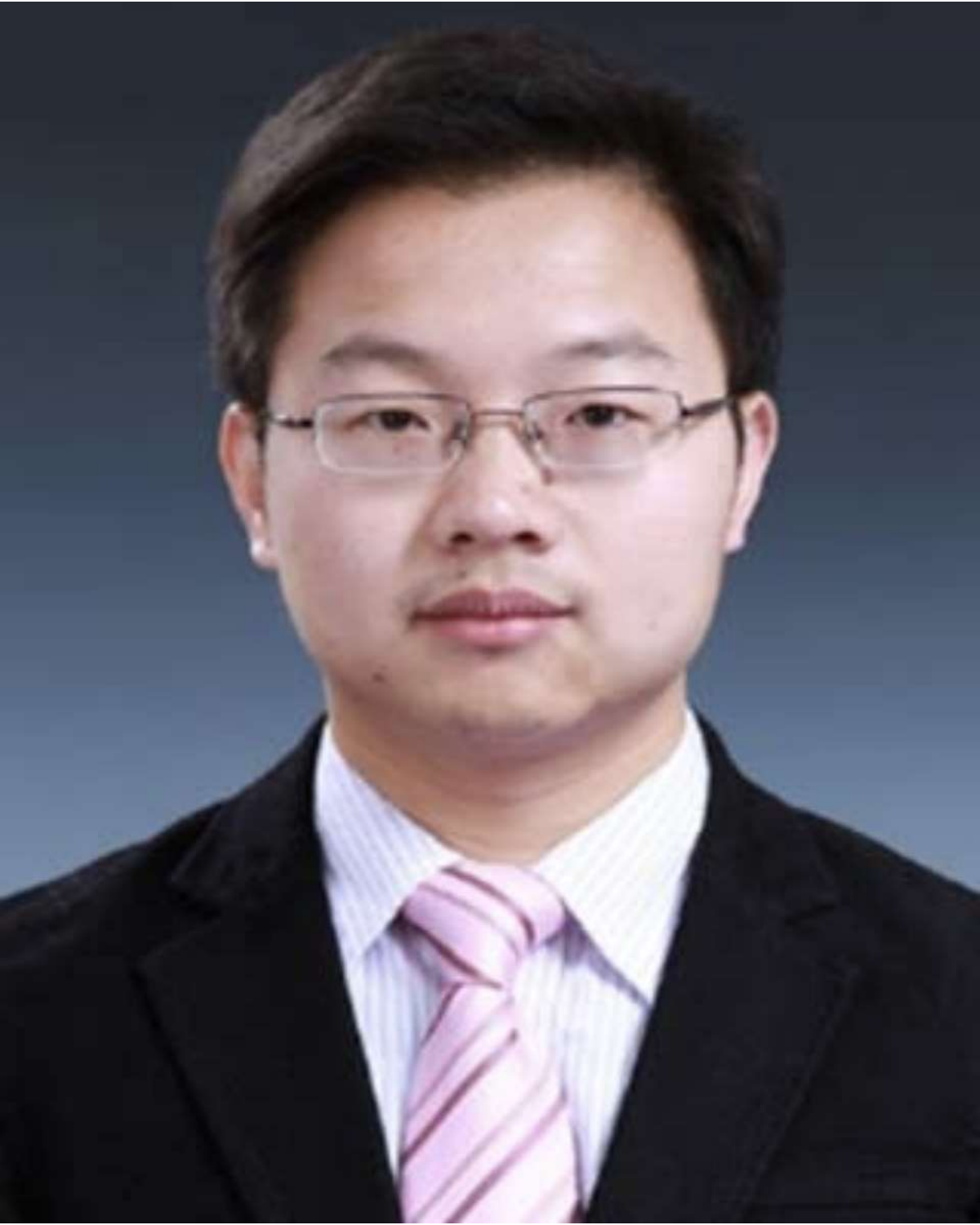}
received the B.S. degree in statistical science from Sun Yat-sen University, Guangzhou, China, in 2007 and the Ph.D. degree in electrical and electronic engineering from Nanyang Technological University (NTU), Singapore, in 2012. 
After briefly working as a Research Fellow at NTU, he joined Tsinghua University, Beijing, China, where he is currently a Full Professor in the Department of Automation. He held visiting positions with Politecnico di Torino, Turin, Italy, Hong Kong University of Science and Technology, Hong Kong, China, University of Melbourne, Melbourne, Victoria, Australia, and so on. His research interests include the intersections between control, optimization and learning, as well as their applications in autonomous systems. 

Dr. You received the Guan Zhaozhi Award at the 29th Chinese Control Conference in 2010 and the ACA (Asian Control Association) Temasek Young Educator Award in 2019. He received the National Science Funds for Excellent Young Scholars in 2017 and for Distinguished Young Scholars in 2023. He is currently an Associate Editor for \textit{Automatica} and IEEE TRANSACTIONS ON CONTROL OF NETWORK SYSTEMS.
	\end{biography}
\newpage
		\begin{biography}{Yang Shi}{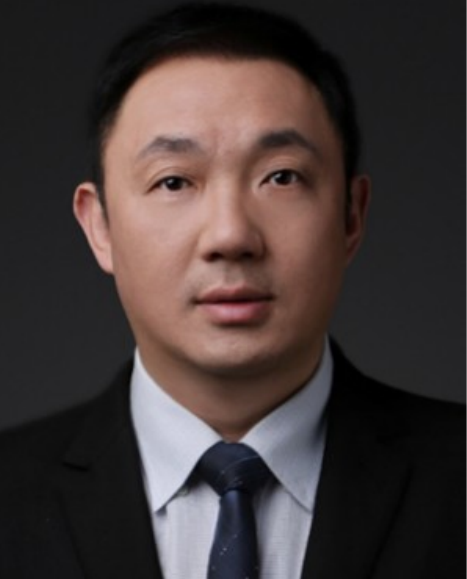}
received the B.Sc. and Ph.D. degrees in mechanical engineering and
automatic control from Northwestern Polytechnical University, Xi’an, China, in 1994 and 1998, respectively, and the Ph.D. degree in electrical and computer engineering from the University of Alberta, Edmonton, AB, Canada, in 2005. He was a Research Associate with the Department of Automation, Tsinghua University, China, from 1998 to 2000. From 2005 to 2009, he was an Assistant Professor and an Associate Professor with the Department of Mechanical Engineering, University of Saskatchewan, Saskatoon, SK, Canada. In 2009, he joined the University of Victoria, and currentlu he is a Professor with the Department of Mechanical Engineering, University of Victoria, Victoria, BC, Canada. His current research interests include networked and distributed systems, model predictive control (MPC), cyber-physical systems (CPS), robotics and mechatronics, navigation and control of autonomous systems (AUV and UAV), and energy system applications.
	
Dr. Shi is the IFAC Council Member. He is a fellow of ASME, CSME,
Engineering Institute of Canada (EIC), Canadian Academy of Engineering (CAE), Royal Society of Canada (RSC), and a registered Professional Engineer in British Columbia and Canada. He received the University of Saskatchewan Student Union Teaching Excellence Award in 2007, the Faculty of Engineering Teaching Excellence Award in 2012 at the University of Victoria (UVic), and the 2023 REACH Award for Excellence in Graduate Student Supervision and Mentorship. On research, he was a recipient of the JSPS Invitation Fellowship (short-term) in 2013, the UVic Craigdarroch Silver Medal for Excellence in Research in 2015, the Humboldt Research Fellowship for Experienced Researchers in 2018, CSME Mechatronics Medal in 2023, the IEEE Dr.-Ing. Eugene Mittelmann Achievement Award in 2023,
the 2024 IEEE Canada Outstanding Engineer Award. He was a Vice-President on Conference Activities of IEEE IES from 2022 to 2025 and the Chair of IEEE IES Technical Committee on Industrial Cyber-Physical Systems. Currently, he is the Editor-in-Chief of IEEE TRANSACTIONS ON INDUSTRIAL ELECTRONICS. He also serves as Associate Editor for \textit{Automatica}, IEEE TRANSACTIONS ON AUTOMATIC CONTROL, and \textit{Annual Review in Controls}.
\end{biography}

		\begin{biography}{Ligang Wu}{Au-WuLG.pdf} received the B.S. degree
in automation from Harbin University of Science
and Technology, Harbin, China, in 2001, and the
M.E. degree in navigation guidance and control
and the Ph.D. degree in control theory and control
engineering from Harbin Institute of Technology,
Harbin, in 2003 and 2006, respectively.

From January 2006 to April 2007, he was a
Research Associate with the Department of Mechanical Engineering, The University of Hong Kong, Hong Kong. From September 2007 to June 2008, he was a Senior Research Associate with the Department of Mathematics, City University of Hong Kong, Hong Kong. From December 2012 to 2013, he was a Research Associate with the Department of Electrical and Electronic Engineering, Imperial College London, London, U.K. In 2008, he joined Harbin Institute of Technology as an Associate Professor and was promoted to a Full Professor in 2012. He has published seven research monographs and more than 170 research articles in internationally refereed journals. His current
research interests include switched systems, stochastic systems, computational and intelligent systems, sliding-mode control, and advanced control techniques for power electronic systems.

Prof. Wu was a recipient of the National Science Fund for Distinguished Young Scholars in 2015 and China Young Five Four Medal in 2016. He was named as a Distinguished Professor of Chang Jiang Scholar in 2017 and has been recognized as a Highly Cited Researcher since 2015. He currently serves as an Associate Editor for several journals, including IEEE TRANSACTIONS ON AUTOMATIC CONTROL, IEEE TRANSACTIONS ON INDUSTRIAL ELECTRONICS, IEEE/ASME TRANSACTIONS ON MECHATRONICS,  \textit{Information Sciences},  \textit{Signal Processing}, and  \textit{IET Control Theory and Applications}. He is also an Associate Editor of the Conference Editorial Board and the IEEE Control Systems Society.
\end{biography}
\end{document}